\documentclass[12pt]{amsart}
\usepackage{amssymb, amsfonts}
\usepackage{hyperref}

\usepackage[retainorgcmds]{IEEEtrantools}

\renewcommand{\>}{\rangle}

\renewcommand{\a}{\alpha}
\renewcommand{\b}{\beta}
\renewcommand{\d}{\delta}

\newcommand{\e}{\varepsilon}
\newcommand{\g}{\gamma}

\renewcommand{\l}{\lambda}

\newcommand{\s}{\sigma}
\renewcommand{\th}{\theta}
\renewcommand{\o}{\omega}
\renewcommand{\O}{\Omega}
\renewcommand{\S}{\Sigma}

\newcommand{\U}{\Upsilon}

\newcommand{\cA}{{\mathcal A}}
\newcommand{\cB}{{\mathcal B}}
\newcommand{\cC}{{\mathcal C}}
\newcommand{\cD}{{\mathcal D}}
\newcommand{\cE}{{\mathcal E}}

\newcommand{\cH}{{\mathcal H}}

\newcommand{\cK}{{\mathcal K}}
\newcommand{\cL}{{\mathcal L}}

\newcommand{\cP}{{\mathcal P}}

\newcommand{\bZ}{{\mathbb Z}}
\newcommand{\bR}{{\mathbb R}}
\newcommand{\bC}{{\mathbb C}}

\newcommand{\Bn}{{\mathcal B^n}}
\newcommand{\Hn}{{\mathcal H^n}}
\newcommand{\Hnk}{{\mathcal H^{k+n}}}
\newcommand{\Hk}{{\mathcal H^{|k|}}}
\newcommand{\Onk}{{\O^n_k}}
\newcommand{\Ukx}{U^{|k|}_x}
\newcommand{\Unk}{U^{k+n}}

\newcommand{\Lat}{{Latr\'emoli\`ere}}

\newtheorem{thm}{Theorem}[section]
\newtheorem{cor}[thm]{Corollary} 
\newtheorem{prop}[thm]{Proposition} 

\newtheorem{keylemma}[thm]{Key Lemma} 
\newtheorem{lem}[thm]{Lemma} 

\theoremstyle{definition}   

\newtheorem{defn}[thm]{Definition} 
\newtheorem{exam}[thm]{Example} 
\newtheorem{notation}[thm]{Notation} 

\numberwithin{equation}{section}

\title
[Vector bundles for matrix algebras]
{Vector bundles for ``Matrix algebras converge to the sphere''  
}

\author{Marc A. Rieffel}
\address{Department of Mathematics \\
University of California \\ Berkeley, CA 94720-3840}
\email{rieffel@math.berkeley.edu}

\thanks{This work is part of the project supported by the grant H2020-MSCA-RISE-2015-691246-QUANTUM DYNAMICS-3542/H2020/2016/2.
}

\dedicatory{In celebration of   
the seventieth birthday of Alain Connes}

\subjclass[2010]{Primary 46L87; 
Secondary 58B34, 81R05, 81R15 
}
\keywords{C*-metric space, 
quantum Gromov-Hausdorff distance, vector bundles,
projective modules, sphere, matrix algebras}

\begin{document}

\begin{abstract}
In the high-energy quantum-physics literature, one finds 
statements such as ``matrix algebras converge to the sphere''.
Earlier I provided a general precise setting for understanding such
statements, in which the matrix algebras are viewed as
quantum metric spaces, and convergence is with respect
to a quantum Gromov-Hausdorff-type distance. 

But physicists want even more to treat structures on spheres (and other spaces), such as vector bundles, Yang-Mills functionals, Dirac operators, etc., and they want to approximate these by corresponding structures on matrix algebras. In the present paper we treat this
idea for vector bundles. We develop a general precise way for understanding how,
for two compact quantum metric spaces that are close together, to a given vector bundle on one of them there can correspond in a natural way a unique vector bundle on the other. 
We then show explicitly how this works 
for the case of matrix algebras converging to the
2-sphere.
\end{abstract}

\maketitle
\allowdisplaybreaks



\tableofcontents


\section{Introduction}
\label{secintro}

In several earlier papers \cite{R7, R25, R21, R29} I showed how to give 
a precise meaning to statements in the literature of high-energy
physics and string theory of the kind ``matrix algebras 
converge to the sphere''. (See the references to the quantum physics
literature given in 
\cite{R7, R17, Stn2, ChZ, DiG, AHI, AcD}.) I did this by introducing
and developing a concept of ``compact quantum metric spaces'', 
and a corresponding quantum Gromov-Hausdorff-type distance 
between them. The compact quantum spaces are unital C*-algebras,
and the metric data is given by equipping the
algebras with suitable seminorms that play the role of the usual Lipschitz
seminorms on the algebras of continuous functions on ordinary 
compact metric spaces. The natural setting for ``matrix algebras 
converge to the sphere'' is that of coadjoint orbits of compact
semi-simple Lie groups, as shown in \cite{R7, R21, R29}.

But physicists need much more than just the algebras. They need vector
bundles, gauge fields, Dirac operators, etc. In the present
paper I provide a general method for giving
precise quantitative meaning to statements in the physics literature of the kind
``here are the vector bundles over the matrix algebras that
correspond to the monopole bundles on the sphere''
\cite{GKP, BBV, Vlt1, Ydr, GRSt, BKV, Stn, Vlt2, Imn, DFHL, Stn2}. 
I then apply this method to the case
of the 2-sphere, with full proofs of convergence. 
Many of the considerations in this paper apply
directly to the general case of coadjoint orbits. But some of the
detailed estimates needed to prove convergence require
fairly complicated considerations (see Section \ref{core2}) 
concerning highest weights for representations of
compact semi-simple Lie groups. It appears to me
that it would be quite challenging to carry
out those details for the general case, though I expect that
some restricted cases, such as matrix-algebra approximations
for complex projective spaces \cite{CSW, Ydr, MrS, CHMO,HlW}, 
are quite feasible to deal with.

In \cite{R17} I studied the convergence of ordinary 
vector bundles on ordinary compact metric spaces for the ordinary 
Gromov-Hausdorff distance. The approach that worked for me
was to use the correspondence between vector bundles
and projective modules (Swan's theorem \cite{Blk}), and
by this means represent 
vector bundles by corresponding projections
in matrix algebras over the algebras of continuous functions
on the compact metric spaces; and then to prove appropriate
convergence of the projections. In the present paper we
follow that same approach, in which now we also
consider projective modules over the matrix
algebras that converge to the 2-sphere, 
and thus also projections in matrix algebras
over these matrix algebras.

For this purpose, one needs
Lipschitz-type seminorms on all of the matrix algebras over the
underlying algebras, with these seminorms coherent in the
sense that they form a ``matrix seminorm'' . In my
recent paper \cite{R29} the theory of
these matrix seminorms 
was developed, and properties 
of such matrix seminorms for the setting of coadjoint orbits were
obtained. In particular, some general methods 
were given for obtaining estimates
related to how these matrix seminorms mesh with an appropriate
quantum analog of the Gromov-Hausdorff distance. The results of that paper
will be used here.

Recently \Lat \ introduced an improved version of quantum
Gromov-Hausdorff distance \cite{Ltr2}, which he calls 
``quantum Gromov-Hausdorff propinquity''. 
In \cite{R29} I
showed that his propinquity works very well 
for our setting of coadjoint orbits,
and so propinquity is the form of quantum Gromov-Hausdorff
distance that we use in the present paper. 
\Lat \ defines his propinquity
in terms of an improved version of the ``bridges'' that I had used
in my earlier papers. For our matrix seminorms we need
corresponding ``matricial bridges''. In \cite{R29}  
natural such matricial
bridges were constructed for the setting of coadjoint orbits.
They will be used here.

To give somewhat more indication
of the nature of our approach, we give now an imprecise version
of one of the general theorems that we apply.
For simplicity of
notation, we express it here only for the C*-algebras, rather than for
matrix algebras over the C*-algebras as needed later.
Let $(\cA, L^\cA)$ and $(\cB, L^\cB)$ be compact quantum 
metric spaces, where $\cA$ and $\cB$ are unital C*-algebras
and $L^\cA$ and $L^\cB$ are suitable seminorms on them. 
Let
$\Pi$ be a bridge between $\cA$ and $\cB$. Then 
$L^\cA$ and $L^\cB$ can be used to measure 
$\Pi$. We denote the resulting length 
of $\Pi$ by $l_\Pi$.
Then we will see, imprecisely speaking, that $\Pi$
together with $L^\cA$ and $L^\cB$ determine a
suitable seminorm, $L_\Pi$, on  $\cA \oplus \cB$.
 
\begin{thm} [Imprecise version of Theorem \ref{th4.5}]
Let $(\cA, L^\cA)$ and $(\cB, L^\cB)$ be compact quantum 
metric spaces,
and let $\Pi$ be a bridge between $\cA$ and $\cB$, with
corresponding  seminorm $L_\Pi$
on  $\cA \oplus \cB$. Let $l_\Pi$ be the length of $\Pi$ as measured
using $L^\cA$ and $L^\cB$. Let $p \in \cA$ and $q \in \cB$ be
projections. If \ 
$
l_\Pi L_\Pi(p,q) < 1/2   , 
$ \
and if $q_1$ is another projection in $\cB$ such that \ 
$
l_\Pi L_\Pi(p,q_1) < 1/2   , 
$ \ 
then there is a continous path, $t \to q_t$ of projections in $\cB$
going from $q$ to $q_1$, so that the projective modules corresponding
to $q$ and $q_1$ are isomorphic. In this way, 
to the projective $\cA$-module
determined by $p$ we have associated a 
uniquely determined isomorphism
class of projective $\cB$-modules.
\end{thm}

In sections \ref{secprojA} through \ref{secproof}, theorems 
of this type are then applied
to the specific situation of matrix algebras converging to
the 2-sphere, in order to obtain our 
correspondence between projective modules
over the 2-sphere and projective modules over the matrix algebras. 

Very recently \Lat \ introduced a fairly different way of saying when 
two projective modules over compact quantum 
metric spaces that are close
together correspond \cite{Ltr6}. For this purpose he equips projective 
modules with seminorms that play the role of a weak analog
of a connection on a vector bundle over a Riemannian manifold, much 
as our seminorms on a C*-algebra 
are a weak analog of the total derivative (or of
the Dirac operator) of a Riemannian manifold. 
As experience is gained with more examples it will be interesting
to discover the relative strengths and weaknesses of these two
approaches.

My next project is to try to understand how the Dirac operator
on the 2-sphere is related to ``Dirac operators'' on the
matrix algebras that converge to the 2-sphere, especially since in the
quantum physics literature there are at least three
inequivalent Dirac operators suggested for the matrix
algebras. This will involve results from \cite{R22}.


\section{Matrix Lip-norms and state spaces}
\label{secmats}

As indicated above, the projections representing projective 
modules are elements of matrix algebras over the basic
C*-algebras.
In this section we give some useful perspective on the relations
between certain types of 
seminorms on matrix algebras over a given C*-algebra
and the metrics on the state spaces of the matrix algebras that
come from the seminorms.

Let $\cA$ be a unital C*-algebra. For a given natural number $d$ 
let $M_d$ denote
the algebra of $d\times d$ matrices with entries in $\bC$, and
let $M_d(\cA)$ denote
the C*-algebra of $d\times d$ matrices with entries in $\cA$. Thus
$M_d(\cA) \cong M_d \otimes \cA$. Since $\cA$ is unital, we
can, and will, identify $M_d$ with the subalgebra 
$M_d \otimes 1_\cA$ of $M_d(\cA)$.

We recall from definition 2.1 of \cite{R29} that by a
``slip-norm" on a unital C*-algebra $\cA$ we mean a
$*$-seminorm $L$ on $\cA$ that is permitted to take
the value $+\infty$ and is such that $L(1_A) = 0$.
Given a slip-norm $L^\cA$ on $\cA$, we will need 
slip-norms, $L_d^\cA$, on each $M_d(\cA)$ that correspond
somewhat to $L^\cA$. 
It is reasonable to want these seminorms
to be coherent
in some sense as $d$ varies. As discussed 
before definition 5.1 of \cite{R29}, the appropriate
coherence requirement is that the sequence
$\{L^\cA_d\}$ forms a ``matrix slip-norm''.
To recall what 
this means, for any positive integers $m$ and $n$ we let
$M_{mn}$ denote  the linear space of $m \times n$ 
matrices with complex entries, equipped with the norm obtained
by viewing such matrices as operators from the Hilbert
space $\bC^n$ to the Hilbert space $\bC^m$. We then note
that for any $A \in M_n(\cA)$, for any $\a \in M_{mn}$,
and any $\b \in M_{nm}$, the
usual matrix product $\a A \b$ is in $M_m(\cA)$. 

\begin{defn}
\label{defmtx}
A sequence $\{L^\cA_d\}$  
is a  \emph{matrix slip-norm} for $\cA$ if 
$L^\cA_d$ is a $*$-seminorm (with value $+\infty$ allowed)
on $M_d(\cA)$ 
for each integer $d \geq 1$, and if this family of
seminorms has the following properties:
\begin{enumerate}
\item For any  $A \in M_d(\cA)$, any $\a \in M_{md}$, and
any $\b \in M_{dm}$, we have
\[
L_m^\cA(\a A \b) \leq \|\a\|L_d^\cA(A)\|\b\|.
\]
\item For any  $A \in M_m(\cA)$ and any  $C \in M_n(\cA)$
we have
\[
L_{m+n}^\cA\left(
\begin{bmatrix}
A & 0 \\
0 & C
\end{bmatrix}
\right)
= \max(L_m^\cA(A), L_n^\cA(C))    .
\]   
\item $L_1^\cA$ is a slip-norm, in the sense that  $L_1^\cA(1_\cA) = 0$.
(But $L_1^\cA$ is also allowed to take value 0 on elements that are
not scalar multiples of the identity.) 
\end{enumerate}
We will say that such a matrix slip-norm is \emph{regular} if
$L^\cA_1(a) = 0$ only for $a \in \bC 1_\cA$.
\end{defn}

The properties above imply that for $d \geq 2$ 
the null-space of $L_d^\cA$ will contain the subalgebra $M_d$,
not just the scalar multiples of the identity,
so that $L_d^\cA$ is a slip-norm . This is why 
our definition of a slip-norm does not require that the
null-space be exactly the scalar multiples of the identity. 
When $\{L^\cA_d\}$ is regular,
the properties above imply that for $d \geq 2$ 
the null-space of $L_d^\cA$ will be exactly $M_d$.

In generalization of the relation between $M_d(\cA)$ and $M_d$,
let now $\cA$ be any unital C*-algebra, and let $\cB$ be a unital
C*-subalgebra of $\cA$ ($1_\cA \in \cB$). We let $S(\cA)$ denote
the state space of $\cA$, and similarly for $S(\cB)$. 
It will be useful for us to view $S(\cA)$ as fibered over $S(\cB)$
in the following way. For any $\nu \in S(\cB)$ let
\[
S_\nu(\cA) = \{\mu \in S(\cA) : \mu | _\cB = \nu\}   .
\] 
Since each $\nu \in S(\cB)$ has at least one extension
to an element of $S(\cA)$, no $S_\nu(\cA)$ is empty,
and so the $S_\nu(\cA)$'s form a partition, or fibration,
of $S(\cA)$ over $S(\cB)$. We will apply this observation
to view $S(M_d(\cA))$ as fibered over $S(M_d)$.

Now let $\{L^\cA_d\}$ be a matrix slip-norm for $\cA$.
For any given $d$ the seminorm $L^\cA_d$ determines
a metric $\rho^{L^\cA_d}$ (with value $+\infty$ permitted,
so often referred to as an ``extended metric'') 
on $S(M_d(\cA))$ defined by
\begin{equation}
\label{eqmatlip}
\rho^{L^\cA_d}(\mu_1, \mu_2) = 
\sup\{|\mu_1(A) - \mu_2(A)|: L^\cA_d(A) \leq 1\}  .
\end{equation}
(Notation like $\rho^{L^\cA_d}$ will be used through most of 
this paper.)
We then observe that if $\mu_1|_{M_d} \neq \mu_2|_{M_d}$
then $\rho^{L^\cA_d}(\mu_1, \mu_2) = +\infty$, 
because there will
exist an $A \in M_d$ such that $L^\cA_d(rA) = 0$ for all
$r \in \bR^+$ while 
$|\mu_1(rA) - \mu_2(rA)| = r|\mu_1(A) - \mu_2(A)| \neq 0$.
Thus $\rho^{L^\cA_d}$ can be finite only on the fibers of
the fibration of $S(M_d(\cA))$ over $S(M_d)$.
 
Consistent with definition 5.1 of \cite{R5} (which treats the
more general case of order-unit spaces) and theorem
1.8 of \cite{R4}, we have:

\begin{defn}
\label{deflip}
Let $\cA$ be a unital C*-algebra. By a \emph{Lip-norm}
on $\cA$ we mean a $*$-seminorm $L$ on $\cA$ (with value
$+ \infty$ allowed) that satisfies
\begin{enumerate}
\item $L$ is semifinite, i.e. $\{a \in \cA:L(a) < \infty\}$ is dense
in $\cA$.
\item For any $a \in \cA$ we have $L(a) = 0$ if and only if
$a \in \bC 1_\cA$.
\item $L$ is lower semi-continuous with respect to the
norm of $\cA$, i.e. for any $r \in \bR^+$ the set
$\{a \in \cA: L(a) \leq r\}$ is norm-closed.
\item The the topology  on $S(\cA)$ from the metric $\rho^L$,
defined much as in equation \ref{eqmatlip} above, 
coincides with the weak-$*$ topology.
This is equivalent to the property that the image of 
\[
\cL_\cA^1 = \{a \in \cA: a^* = a, \ L(a) \leq 1\} 
\] 
in $\tilde \cA = \cA / \bC  1_\cA$
is totally bounded for the quotient norm $\| \cdot \|\hspace{2pt}\tilde{}$ on 
$\tilde \cA$. (Or, equivalently, that the image in $\tilde \cA$ of
$\{a \in \cA: L(a) \leq 1\}$ is totally bounded.) 
\end{enumerate}
\end{defn}

\begin{defn}
\label{defmlip}
Let $\cA$ be a unital C*-algebra. By a \emph{ matrix Lip-norm}
on $\cA$ we mean a matrix slip-norm $\{L^\cA_d\}$ for $\cA$
which has the property that $L^\cA_1$ is a Lip-norm for $\cA$
and each  $L^\cA_d$ is lower semi-continuous.   
\end{defn}

We remark that from property (1) of Definition \ref{defmtx} it
follows then that each $L^\cA_d$ is semi-finite.

\begin{prop}
\label{proptb}
Let $\cA$ be a unital C*-algebra and let $\{L^\cA_d\}$ be a
matrix Lip-norm on $\cA$. For each natural number $d$ let
\[
\cL^1_{M_d(\cA)} = \{A \in M_d(\cA): A^* = A, \ L^\cA_d(A) \leq 1\}.
\] 
Then the image of $\cL^1_{M_d(\cA)} $ in the quotient $M_d(\cA) / M_d$
is totally bounded for the quotient norm.
\end{prop} 

\begin{proof}
Let $A \in \cL^d_1$, with $A$ the matrix $\{a_{jk}\}$. Then 
$L_d^\cA(A) \leq 1$, and so by property (1) of Definition \ref{defmtx}
we have $L_1^\cA(a_{jk}) \leq 1$ for all $j, \ k$. Thus for each
fixed pair $(j, \ k)$ the set of $(j, \ k)$-entries of all the elements
of $\cL^1_{M_d(\cA)} $ lie in $\{a\in\cA:L(a) \leq 1\}$, 
whose image in $\tilde \cA$ is
totally bounded. But the finite product of totally bounded sets is
totally bounded for any of the equivalent natural metrics
on the product.
\end{proof}

\begin{prop}
\label{propweak}
Let $\cA$ be a unital C*-algebra and let $\{L^\cA_d\}$ be a
matrix Lip-norm on $\cA$. For each natural number $d$
and each $\nu \in S(M_d) $ the topology on
the fiber $S_\nu(M_d(\cA))$
in $S(M_d(\cA))$ determined by the restriction to 
$S_\nu(M_d(\cA))$ of the metric $\rho^{L^\cA_d}$,
agrees with the weak-$*$ topology restricted to
$S_\nu(M_d(\cA))$ (and so $S_\nu(M_d(\cA))$ is compact). 
\end{prop}

\begin{proof}
This is an immediate corollary of theorem 1.8 of \cite{R4}
and Proposition \ref{proptb}
when one lets $M_d(\cA)$ be the normed space $A = \cL$ of
theorem 1.8 of \cite{R4}, lets $L^\cA_d$ be the $L$ of that theorem,
lets $M_d$ be the subspace $\cK$ of that
theorem, and lets the state $\nu$ be the $\eta$ of that theorem. 
\end{proof}

As a consequence, even though $\rho^{L^\cA_d}$ can take
value $+\infty$, it is reasonable to talk about
whether a subset $Y$
of $S(M_d(\cA))$ is $\e$-dense in $S(M_d(\cA))$. This
just means that, as usual, for each $\mu \in S(M_d(\cA))$
there is an element of $Y$ that is within distance $\e$ of it.
This observation will shortly be of importance to us.

A good class of simple examples to keep in mind for all of this
is given next. It is the class that is
central to the paper \cite{R17}.

\begin{exam}
\label{excomm}
Let $(X, \rho)$ be a compact metric space, and let
$\cA = C(X)$. For a fixed natural number $d$ let $L^\cA_d$
be defined on $M_d(\cA) = C(X, M_d)$ by
\[
L^\cA_d(F) = \sup\{\|F(x) - F(y)\|/\rho(x, y): x, y \in X \quad \mathrm{and}
\quad x \neq y \}
\]  
for $F \in M_d(\cA)$. Then for $d \geq 2$ the metric on $S(M_d(\cA))$
determined by $L^\cA_d$ will take on value $+\infty$. But $S(M_d(\cA))$
will be fibered over $S(M_d)$ and the metric will be finite on each
fiber, and the topology it determines on the fiber will coincide with the
weak-$*$ topology there.
\end{exam}

Different ways of dealing with seminorms that may have a large
null-space can be found in definitions 2.1 and 2.3 of \cite{Ltr1a} 
and in definition 2.3 of \cite{AgB}, but they do not seem to
be useful for our present purposes.


\section{Quotients of C*-metric spaces}
\label{secquot}

Let $(Z, \rho_Z)$ be a compact metric space, and let $X$ be
a closed subset of $Z$. Let $\cC = C(Z)$ and let
$\cA = C(X)$. By restricting functions on $Z$ to
the subset $X$ we see that $\cA$ is a quotient
algebra of $\cC$. We need to consider the corresponding
non-commutative situation. We will mostly use it
for the non-commutative analog of the situation in
which $(X, \rho_X)$ and $(Y, \rho_Y)$ are two compact
metric spaces and $Z $ is the disjoint union of
$X$ and $Y$ (with $\rho_Z$ compatible with $\rho_X$
and $\rho_Y$). Then $\cC = \cA \oplus \cB$ where
$\cB = C(Y)$. We will need the matricial version of
this situation.

Accordingly, let $\cA$ and $\cC$ be unital C*-algebras, 
and let $\pi$ be a surjective $*$-homomorphism from
$\cC$ onto $\cA$, so that 
$\cA$ is a quotient of $\cC$. Then by composing with
$\pi$, every state of $\cA$ determines a state of $\cC$.
In this way we obtain a continuous injection of
$S(\cA)$ into $S(\cC)$, and we will often just
view $S(\cA)$ as a subset of $S(\cC)$ without
explicitly mentioning $\pi$.

We think of $\cA$ and $\cC$ as possibly being matrix
algebras over other algebras, and so we will consider
a slip-norm, $L$, on $\cC$, not requiring that $L$
take value 0 only on $\bC 1_\cC$. Thus the metric
$\rho^L$ on $S(\cC)$ can take value $+\infty$,
but we can consider the situation in which, nevertheless,
$S(\cA)$ is $\e$-dense in $S(\cC)$ for some given
$\e \in \bR^+$.

The main result of this section is the following
proposition, which is a generalization of 
key lemma 4.1 of \cite{R17}. A related result in
a more restricted setting, relevant to
our next section, is emphasized in the
paragraph before remark 6.5 of \cite{Ltr2}. The inequality
obtained in our proposition will be basic for
later sections of this paper.

\begin{prop}
\label{prpkey}
Let $\cA$ and $\cC$ be unital C*-algebras, 
and let $\pi$ be a surjective $*$-homomorphism of
$\cC$ onto $\cA$, so that $S(\cA)$ can be viewed
as a subset of $S(\cC)$. Let $L$ be a slip-norm
on $\cC$, and let there be given $\e \in \bR^+$.
If $S(\cA)$ is $\e$-dense in $S(\cC)$ for $\rho^L$,
then for any $c \in \cC$ satisfying $c^* = c$ we have
\[
\|c\| \leq \| \pi(c) \| + \e L(c)  .
\]
\end{prop}

\begin{proof}
Let  $c \in \cC$ satisfy $c^* = c$. Then there
is a $\mu \in S(\cC)$ such that $|\mu(c)| = \|c\|$.
By assumption there is a $\nu \in S(\cA)$ such
that $\rho^L(\nu \circ \pi, \ \mu) \leq \e$. This 
implies that 
\[
|\nu(\pi(c)) - \mu(c)| \leq \e L(c),
\]
so that
\[
\|c\| = |\mu(c)| \leq |\nu(\pi(c)| + \e L(c) \leq \|\pi(c)\| + \e L(c)  ,
\]
as needed.
\end{proof}

It would be interesting to know whether the converse of this
proposition is true, that is, whether the inequality implies
the $\e$-denseness. It is true for ordinary compact
metric spaces.


\section{Bridges and $\e$-density}
\label{secvbclose}

We now recall how in \cite{R29} we used slip-norms
in connection with \Lat's bridges so that we are able to
deal also with matricial bridges. Let $\cA$ and $\cB$ be
unital C*-algebras, and let
$\Pi = (\cD, \o)$ be a bridge from $\cA$ to $\cB$ in the sense 
of {\Lat}   \cite{Ltr2, Ltr5}. That is, $\cA$ and $\cB$ are 
identified as  C*-subalgebras
of the C*-algebra $\cD$ that each contain $1_\cD$, 
while $\o \in \cD$, 
$\o^* = \o$,   $\|\o\| \leq 1$, and
$1 \in \s(\o)$ (which is more than {\Lat} requires, but which holds 
for our main examples). \Lat \ calls $\o$ the ``pivot'' of
the bridge. The specific
bridges that we will use for the case of $SU(2)$ are described
in Section \ref{secalg}.

Fix a positive integer $d$. We can view $M_d(\cA)$ 
and $M_d(\cB)$ as unital
subalgebras of $M_d(\cD)$. Let
$\o_d = I_d \otimes \o$, where $I_d$ is the identity element
of $M_d$, so $\o_d$ can be viewed as the diagonal
matrix in $M_d(\cD)$ with $\o$ in each diagonal entry.
Then it is easily seen that $\Pi_d = (M_d(\cD), \o_d)$ is a bridge
from $M_d(\cA)$ to $M_d(\cB)$. 

\begin{defn}
\label{defmat}
For each
natural number $d$ let $\o_d = I_d \otimes \o$.
The bridges $\Pi_d = (M_d(\cD), \o_d)$ are called the 
\emph{matricial bridges}
determined by the bridge $\Pi$.
\end{defn}

Let $L^\cA$
and $L^\cB$ be Lip-norms on $\cA$ and $\cB$. 
\Lat \ defines \cite{Ltr2, Ltr5}
how to use them to measure bridges from $\cA$ to $\cB$.
We recall here how in section 2 of \cite{R29} I 
adapted his
definitions to the case of matricial bridges, using
matrix Lip-norms. Let  $\{L^\cA_d\}$ be a matrix Lip-norm
on $\cA$ and let $\{L^\cB_d\}$ be a matrix Lip-norm
on $\cB$. Fix $d$. Then $L^\cA_d$ is a slip-norm
on $M_d(\cA)$ and $L^\cB_d$ is a slip-norm
on $M_d(\cB)$, and they can be used to measure
the bridge $\Pi_d$, by making only minor modifications
to \Lat's definition. We review how this is done, but
for notational simplicity we 
will not restrict attention to matrix algebras
over algebras. Instead we will work with 
general unital C*-algebras $\cA$ and $\cB$, but
we will use slip-norms on them. 
So, let $\cA$ and $\cB$ be equipped with slip-norms $L^\cA$
and $L^\cB$. We use these slip-norms 
to measure a bridge $\Pi$ from $\cA$ to $\cB$, as follows.

Set, much as before,
\[
\cL^1_\cA = \{a \in \cA: a^* = a \ \ \mathrm{and} \ \ L^\cA(a) \leq 1\},
\]
and similarly for $\cL^1_\cB$. We view these as subsets of $\cD$.
\begin{defn}
\label{reach}
The \emph{reach} of $\Pi$ is defined by:
\[
\mathrm{reach}(\Pi) 
= \mathrm{Haus}_\cD\{\cL_\cA^1  \o \ , \     \o  \cL_\cB^1\}   ,
\]
where $\mathrm{Haus}_\cD$ denotes the Hausdorff distance with respect
to the norm of $\cD$, and where the product defining $\cL_\cA^1  \o$ and
$ \o  \cL_\cB^1$ is that of $\cD$. We will often write $r_\Pi$ for
$\mathrm{reach}(\Pi) $. Note that $r_\Pi$ can be $+\infty$.
\end{defn}

We now show that when a slip-norm is part of a matrix
Lip-norm, as defined in Definition \ref{defmlip}, its
reach is always finite. By definition, the metric on the
state space determined by a Lip-norm gives
the weak-$*$ topology. Since the state
space is compact, it therefore has finite
diameter for the metric. Given a unital C*-algebra
$\cA$ and a Lip-norm $L^\cA$ on it, we denote
the diameter of $S(\cA)$ for the corresponding metric, 
$\rho^\cA$, by
$\mathrm{diam}(\cA)$ (not mentioning $L^\cA$ unless
confusion may arise,
as is common practice).

\begin{lem}
\label{lemfr}
Let $\cA$ be a unital C*-algebra, and let
 $L^\cA$ be a Lip-norm
on $\cA$. Let $\nu$ be a state of $\cA$. For
any $a \in \cA$ we have
\[
\|a - \nu(a)1_\cA\| \leq 2 \ \mathrm{diam}(\cA) L^\cA(a).
\]
\end{lem}

\begin{proof}
Suppose first that $a \in \cA$ with $a^* = a$.
For any state $\mu$ on $\cA$ we have
\[
|\mu(a - \nu(a)1_\cA)| = |\mu(a) - \nu(a)| \leq  \rho^\cA(\mu, \nu) L^\cA(a)
\leq  \mathrm{diam}(\cA) L^\cA(a).
\]
Consequently $\|a - \nu(a)1_\cA\| \leq  \mathrm{diam}(\cA) L^\cA(a)$.
For general $a \in \cA$, when we apply this inequality to the real and
imaginary parts of $a$ we obtain the desired result.
\end{proof}

\begin{prop}
\label{propfr}
Let $\cA$ and $\cB$ be unital C*-algebras,
and let $\Pi = (\cD, \o)$ be a bridge from $\cA$ to $\cB$. Let
$\{L^\cA_d\}$ be a matrix Lip-norm
on $\cA$ and let $\{L^\cB_d\}$ be a matrix Lip-norm
on $\cB$. Let $\mathrm{diam}(\cA)$ be the diameter
of $\cA$ for the Lip-norm $L^\cA_1$, and similarly
for $\cB$. Then for any natural number $d$ we
have
\[
r_{\Pi_d} \leq 2 \ d \ \max\{\mathrm{diam}(\cA), \ \mathrm{diam}(\cB)\}.
\]
\end{prop}

\begin{proof}
By definition, $1 \in \s(\o)$, so we can find a $\psi \in S(\cD)$
such that $\psi(\o) = 1$. We fix such a $\psi$.
Let $d$ be given. Let $A \in M_d(\cA)$  with $A^* = A$ and
$L_d^\cA(A) \leq 1$, and $A = \{a_{jk}\}$. Define $B= \{b_{jk}\}$ 
in $M_d(\cB)$ by $b_{jk} = \psi(a_{jk})1_\cD$. 
Clearly $B^* = B$, and $L_d(B) = 0$ by conditions 1 and 2
of Definition \ref{defmtx}. Then
\begin{align*}
\|A\o_d - \o_d B\| &= \|\{a_{jk}\o - \o \psi(a_{jk}\}\| 
= \|(A-B)\o_d\|  \\
&\leq \|A-B\| \leq d \max \{\|a_{jk} - \psi |_\cA (a_{jk})\|   \\
&\leq 2 \ d \ \mathrm{diam}(\cA)\max \{ L_1^\cA(a_{jk})\}
\leq 2 \ d \ \mathrm{diam}(\cA),
\end{align*}
where for the next-to-last inequality we have used
Lemma \ref{lemfr}, and for the last inequality we
have used condition 1 of Definition \ref{defmtx}
and the fact that $ L_d^\cA(A)\leq 1$. In this way
we see that $A\o_d$ is within 
distance $2d \ \mathrm{diam}(\cA)$ of
$\o_d   \cL^1_{M_d(\cB)}$. On reversing
the roles of $A$ and $B$, we see that we have the
desired result.
\end{proof}

To define the height of $\Pi$ we need to 
consider the state space, $S(\cA)$,
of $\cA$, and similarly for $\cB$ and $\cD$. Even more, we set
\[
S_1(\o) = \{\phi \in S(\cD): \phi(\o) = 1\}     ,   
\]
the ``level-1 set of $\o$''. 
 It is not empty because by
assumption $1 \in \s(\o)$. 
The elements of $S_1(\o)$ are ``definite'' 
on $\o$ in the sense \cite{KR1} that 
for any $\phi \in S_1(\o)$ and $d \in \cD$ we have
\[
\phi(d\o) = \phi(d) = \phi(\o d)     .
\]
Let $\rho^\cA$ denote the metric on $S(\cA)$ determined by $L^\cA$,
defined, much is in equation \ref{eqmatlip}, by
\begin{equation}
\label{metr}
\rho^\cA(\mu, \nu) = \sup\{|\mu(a) - \nu(a)|: a \in \cL^1_\cA  \}.
\end{equation}
(Since we now are not assuming we have Lip-norms, we must
permit $\rho^\cA$ to take the value $+\infty$. Also, it is
not hard to see that the supremum can be taken equally well
just over all of $\{a \in \cA: L^\cA(a) \leq 1\}$.)
Define $\rho^\cB$ on $S(\cB)$ similarly. 
\begin{notation}
We denote by $S_1^\cA(\o)$ the set of restrictions of the
elements of $S_1(\o)$ to $\cA$. We define $S_1^\cB(\o)$ 
similarly. 
\end{notation}

\begin{defn}
\label{height}
Let $\cA$ and $\cB$ be unital C*-algebras and let
$\Pi = (\cD, \o)$ be a bridge from $\cA$ to $\cB$ . 
Let $L^\cA$ and $L^\cB$ be slip-norms on $\cA$ and $\cB$. 
The \emph{height} of the bridge $\Pi$ is given by
\[
\mathrm{height}(\Pi) =
\max\{\mathrm{Haus}_{\rho^\cA}(S_1^\cA(\o), S(\cA)) , \ 
\mathrm{Haus}_{\rho^\cB}(S_1^\cB(\o) , S(\cB))\}  ,
\]
where the Hausdorff distances are with respect to the indicated
metrics determined by $L^\cA$ and $L^\cB$
(with value $+\infty$ allowed). We will 
often write $h_\Pi$ for
$\mathrm{height}(\Pi) $.  
The length of $\Pi$ is then defined by
\[
\mathrm{length}(\Pi) 
= \max\{\mathrm{reach}(\Pi), \ \mathrm{height}(\Pi)\}  .
\]
\end{defn}

Up to now I have not found a proof of the analog for
height of Proposition \ref{propfr}, namely that when
matrix Lip-norms are involved the height is always
finite, though I suspect that this is true. For the 
``bridges with conditional expectation'' that we will use
later (with matrix Lip-norms) the height, and so the length, 
is always finite.

Anyway, we will now just make the quite strong assumption that 
$\mathrm{length}(\Pi) < \infty$. It is shown in section 6 of \cite{R29}
that this assumption is satisfied for the specific class of examples
that we deal with in the present paper. This will be somewhat 
reviewed later in Section \ref{secproof}. 
The main consequence of this assumption for
our present purposes is a generalization to our 
present non-commutative
setting of key lemma 4.1 of \cite{R17}. 
This generalization will yield for this case the same
inequality as just found in Proposition \ref{prpkey} but with the
added information of a relevant value for $\e$. The core
calculations for
this generalization can essentially be found
in the middle of the proof of proposition 5.3 of \cite{Ltr2}. 
We will call our generalization again the Key Lemma. The set-up
is as follows. As above, let $\cA$ and $\cB$ be unital C*-algebras, and
let $\Pi = (\cD, \o)$ be a bridge between them. Let $L^\cA$ and
$L^\cB$ be slip-norms on  $\cA$ and $\cB$, and let $r_\Pi$
and $h_\Pi$ denote the reach and height of $\Pi$ as measured
by $L^\cA$ and $L^\cB$. Assume that $r_\Pi$
and $h_\Pi$ are both finite. Define a seminorm, $N_\Pi$,
on $\cA \oplus \cB$ by
\[
N_\Pi(a, b) = \|a\o \ - \ \o b\|.
\]
Notice that $N_\Pi$ is in general not a $*$-seminorm. Much
as in theorem 6.2 of \cite{R21}, define a $*$-seminorm, $\hat N_\Pi$,
on $\cA \oplus \cB$ by
\[
\hat N_\Pi(a, b) = N_\Pi(a,b) \vee N_\Pi(a^*, b^*)   ,
\]
where $\vee$ means ``maximum''. Of course, $\hat N_\Pi$
agrees with $N_\Pi$ on self-adjoint elements.

Let $r \geq r_\Pi$ be chosen. (The 
reason for not just taking $r = r_\Pi$ will be given in the fifth 
paragraph of the proof of Theorem \ref{mainthm}.) 
Define a seminorm, $L^r$, on the C*-algebra $\cA \oplus \cB$ by
\begin{equation}
\label{eqlr}
L^r(a,b) = L^\cA(a) \vee L^\cB(b) \vee r^{-1} \hat N_\Pi(a, b)   .
\end{equation}
Then $L^r$ is a slip-norm on $\cA \oplus \cB$, and it determines a metric,
$\rho^{L^r}$, on $S(\cA \oplus \cB)$. Note that $\cA$ and $\cB$
are both quotients of $\cA \oplus \cB$ in an evident way, so that we
can consider the quotient seminorms on them coming from $L^r$.
As discussed around example 5.4 of \cite{R21}, there are complications
with quotients of $*$-seminorms on non-self-adjoint elements.
Accordingly, much as for notation 5.5 of \cite{R21}, we make:

\begin{defn}
\label{admisss}
Let $\cA$ and $\cB$, and $L^\cA$ and $L^\cB$ be as above.
We say that a $*$-seminorm $L$ on $\cA \oplus \cB$ is
\emph{admissible} for $L^\cA$ and $L^\cB$ if its quotient on $\cA$
agrees with $L^\cA$ on self-adjoint elements of $\cA$, and
similarly for its quotient on $\cB$.
\end{defn}

\begin{prop}
\label{admis}
With notation as above, if $r \geq r_\Pi$ then $L^r$ is
admissible for $L^\cA$ and $L^\cB$.
\end{prop}

The proof of this proposition is implicit in the proof of theorem 6.3
of \cite{Ltr2}, and amounts to showing that $(a,b) \to r^{-1}N_\Pi(a, b)$,
when restricted to self-adjoint elements,
is a ``bridge'' in the more primitive sense defined in definition 5.1
of \cite{R6}, and then using the main part of the proof of theorem 5.2
of \cite{R6}. For the reader's convenience we give a short
direct proof here, in particular because we will need related facts
in Section \ref{secproof}.

\begin{proof}
Clearly  $L^r(a,b) \geq L^\cA(a)$ for every $b$. It follows
that the quotient of $L^r$ on $\cA$ is no smaller than $L^\cA$.
Let $a \in \cA$ with $a^* = a$ and $L^\cA(a) = 1$. 
Let $\e > 0$ be given. By the definition of $r_\Pi$
there is a $b \in \cB$ with $b^* = b$ and $L^\cB(b) \leq 1$ such that
$\|a\o - \o b\| \leq r_\Pi + \e$. Since $r \geq r_\Pi$ it follows that
\[
L^\cB(b) \vee r^{-1}\|a\o - \o b\|  \ \leq \ 1 + r^{-1}\e = L^\cA(a) + r^{-1}\e.
\]
Since $\e$ is arbitrary, it follow that on $a$ the quotient of
$L^r$ is equal to $L^\cA(a)$. By scaling it follows that  
the quotient of $L^r$ on $\cA$ agrees with $L^\cA$ on
all self-adjoint elements. Reversing the
roles of $\cA$ and $\cB$, we obtain the corresponding fact
for the quotient of $L^r$ on $\cB$.
\end{proof}

The following lemma, which is closely related to 
the comments in the
paragraph before remark 6.5 of \cite{Ltr2},
shows how Proposition \ref{prpkey} is 
relevant to the context of bridges.

\begin{keylemma}
\label{keylem}
With notation as above, let $(a,b) \in \cA \oplus \cB$ with $a^* = a$
and $b^* = b$. Let $r \geq r_\Pi$ be chosen, and let $L^r$ be defined
on $\cA \oplus \cB$ by equation \ref{eqlr}. Then
\[
\|(a,b)\| \leq \|a\| \ + \ (h_\Pi + r)L^r(a,b),
\]
and similarly with the roles of $a$ and $b$ interchanged.
\end{keylemma} 

\begin{proof}
By scaling, it suffices to prove this under the assumption that $L^r(a,b) = 1$,
so we assume this. Let $\nu \in S(\cB)$. By the definition of $h_\Pi$
there is a $\psi \in S_1(\o)$ such that $\rho_{L^\cB}(\nu, \psi |_\cB) \leq h_\Pi$. 
Then, since $\psi$ is definite on $\o$, 
and $L^\cB(b) \leq L^r(a, b) \leq 1$, we have:
\begin{align*}
|\nu(b)| &\leq |\nu(b) - \psi(b)| + |\psi(\o b)|   \\
&\leq \rho_{L^\cB}(\nu, \psi |_\cB)L^\cB(b) 
+ |\psi(\o b) - \psi(a\o)| + |\psi(a\o)|  \\
&\leq h_\Pi + \|a\o \ - \o b\| + |\psi(a)|  \\
&\leq h_\Pi + r +\|a\| .
\end{align*}
Since this holds for all $\nu \in S(\cB)$, and since $b^* = b$, 
it follows that
$
\|b\| \leq \| a \| + h_\Pi + r,
$
and so
\[
\|(a,b)\| \leq \| a \| + h_\Pi + r,
\]
as needed.
\end{proof}


\section{Projections and Leibniz seminorms}
\label{secproleib}

We now assume that the slip-norms 
$L^\cA$ and $L^\cB$ on $\cA$ and $\cB$ are lower semi-continuous
with respect to the norm topologies on $\cA$ and $\cB$. It is then
clear that $L^r$, as defined in equation \ref{eqlr}, 
is lower semi-continuous on $\cA \oplus \cB$
since $N_\Pi$ is norm-continuous on $\cA \oplus \cB$. We now also
assume that $L^\cA$ and $L^\cB$ satisfy the Liebniz inequality,
that is,
\[
L^\cA(aa') \leq L^\cA(a)\|a'\| + \|a\|L^\cA(a')
\]
for any $a, a' \in \cA$, and similarly for $\cB$. We need to assume
in addition that $L^\cA$ and $L^\cB$ are \emph{strongly} Leibniz,
that is, that if $a \in \cA$ is invertible in $\cA$, then
\[
L^\cA(a^{-1}) \leq \|a^{-1}\|^2 L^\cA(a),
\]
and similarly for $\cB$. Then a simple computation, discussed
at the beginning of the proof of theorem 6.2 of \cite{R21},
shows that $L^r$ is strongly Leibniz on $\cA \oplus \cB$.
We will also assume that $L^\cA$ and $L^\cB$ are semi-finite
in the sense that $\{a \in \cA: L^\cA(a) <  \infty\}$ is dense
in $\cA$, and similarly for $\cB$. Then $L^r$ is also semi-finite.

With the above structures as motivation,
we now adapt to our non-commutative setting
many of the basic results of sections 2, 3 and 4 of \cite{R17}.
For notational simplicity 
we first consider a unital C*-algebra $\cC$ (such
as $\cA \oplus \cB$) equipped with a semi-finite 
lower-semi-continuous strongly-Leibniz slip-norm 
$L^\cC$ (defined on all of $\cC$). 

We now consider the relation between the strong Leibniz property and the
holomorphic functional calculus, along the lines of section 2 of
\cite{R17}. Let $c \in \cC$  
and let $\th$ be a $\bC$-valued function
defined and holomorphic in some neighborhood of the spectrum,
$\s(c)$, of $c$.  In the standard way used for ordinary Cauchy integrals,
we let 
$\g$ be a collection of piecewise-smooth oriented closed curves in the
domain 
of $\th$ that surrounds $\s(c)$ but does not meet $\s(c)$, such that
$\th$ on $\s(c)$ is represented by its Cauchy integral using $\g$.
Then 
$z \mapsto (z-c)^{-1}$ will, on the range of $\g$, be a well-defined 
and continuous function with values in $\cC$.  Thus we can define 
$\th(c)$ by
\[
\th(c) = \frac {1}{2\pi i} \int_{\g} \th(z)(z-c)^{-1}dz.
\]
For a fixed neighborhood of $\s(c)$ containing the range of $\g$ 
the mapping $\th \mapsto \th(c)$
is a unital homomorphism 
from the algebra of holomorphic functions on this 
neighborhood of $\s(c)$ into $\cC$ \cite{KR1, Rdn}. The following
proposition is the generalization of proposition 2.3 of \cite{R17}
that we need here.

\begin{prop}
\label{prop2.3}
Let  $L^\cC$ be a lower-semicontinuous strongly-Leibniz 
slip-norm on $\cC$.
For $c \in \cC$, and for $\th$ and $\g$ as above, 
we have  
\[
L^\cC(\th(c)) \le \left( \frac {1}{2\pi} 
\int_{\g} |\th(z)|d|z|\right)(M_{\g}(c))^2L^\cC(c)  ,
\]
where $M_{\g}(c) = \max\{\|(z-c)^{-1}\|: z \in \mathrm{range}(\g)\}$.
\end{prop}

\begin{proof}
It suffices to prove this when $L^\cC(c) < \infty$. Because $L^\cC$ is lower-semicontinuous, it can be brought within the
integral 
defining $\th(c)$, with the evident inequality.  (Think of
approximating 
the integral by Riemann sums.)  Because $L^\cC$ is strongly Leibniz, this
gives
\[
L^\cC(\th(c)) \le \frac {1}{2\pi} \int_{\g} |\th(z)| \ 
\|(z-c)^{-1}\|^2 \ L^\cC(c) \ d|z|.
\]
On using the definition of $M_{\g}(c)$ we obtain the desired
inequality.
\end{proof}

This proposition shows that $\{c \in \cC: L^\cC(c)<\infty\}$ is closed
under the holomorphic functional calculus (and is a dense
$*$-subalgebra of $\cC$ as seen earlier). 
The next proposition is essentially
proposition 3.1 of \cite{R17}. It is a known result
(see, e.g., section~$3.8$ of \cite{GVF}). 
We do not repeat
here the proof of it given in \cite{R17}. 

\begin{prop}
\label{prop3.1}
Let $\cC$ be a unital $C^*$-algebra, and let $C'$ be a dense
$*$-subalgebra 
closed under the holomorphic functional calculus in $\cC$.  Let $p$ be a 
projection in $\cC$.  Then for any $\d > 0$ there is a projection $p_1$
in 
$C'$ such that $\|p-p_1\| < \d$.  If $\d < 1$ then $p_1$ is homotopic
to 
$p$ through projections in $\cC$, that is, there is a continuous
path of projections in $\cC$ going from $p_1$ to $p$.  
\end{prop}

We apply this result to $\{c \in \cC: L^\cC(c) <  \infty\}$. The next
proposition is almost exactly proposition 3.3 of \cite{R17}. We will
not repeat the proof here. It involves Proposition \ref{prop2.3} 
and a mildly complicated
argument involving contour integrals.

\begin{prop}
\label{prop3.3}
Let $\cC$ be a unital C*-algebra
and let $L^\cC$ be a 
lower-semicontinuous strongly-Leibniz slip-norm on $\cC$.  Let
$p_0$ and 
$p_1$ be two projections in $\cC$.  
Suppose that $\|p_0-p_1\| \le \d < 1$, so 
that there is a norm-continuous path, $t \mapsto p_t$, of projections
in $\cC$ 
going from $p_0$ to $p_1$ \cite{Blk, RLL}.  If $L^\cC(p_0) < \infty$ 
and $L^\cC(p_1) < \infty$,
then we can arrange that 
\[
L^\cC(p_t) \le (1-\d)^{-1} \max\{L^\cC(p_0),L^\cC(p_1)\}
\]
for every $t$.  
\end{prop}

We now let $\cA$ be a unital C*-algebra that is
a quotient of $\cC$, with $\pi: \cC \to \cA$ the quotient map (such as
the evident quotient map from our earlier  $\cA \oplus \cB$ onto $\cA$).
We let $L^\cA$ be the quotient of $L^\cC$ on $\cA$, and we assume
that $L^\cA$ is semi-finite, lower semi-continuous, 
and strongly Leibniz (which is not 
automatic --- see section 5 of \cite{R21}). 
Motivated by Key Lemma \ref{keylem}, we will be making hypotheses such
as that there is an $\e > 0$ (such as $h^\Pi + r$) such that 
\[
\|c\| \leq \|\pi(c)\| + \e L^\cC(c)
\]
for all $c \in \cC$. The next
proposition is our non-commutative version of theorem 4.2
of \cite{R17}.

\begin{thm}
\label{thm4.2}
Let $\cC, \cA, \pi, L^\cC$,and $L^\cA$ be as above. Suppose given
an $\e > 0$ such that for all $c \in \cC$ with $c^* = c$ we have
\[
\|c\| \leq \| \pi(c)\| + \e L^\cC(c).
\]
 Let $p_0$ and $p_1$ be
projections 
in $\cA$, and let $q_0$ and $q_1$ be projections in $\cC$
such 
that $\pi(q_0) = p_0$ and $\pi(q_1) = p_1$.  Set
\[
\d = \|p_0-p_1\| + \e(L^\cC(q_0) + L^\cC(q_1)).
\]
If $\d < 1$, then there is a path, $t \mapsto q_t$, through
projections 
in $\cC$, from $q_0$ to $q_1$, such that
\[
L^\cC(q_t) \le (1-\d)^{-1} \max\{L^\cC(q_0),L^\cC(q_1)\}
\]
for all $t \in [0,1]$.  
\end{thm}

\begin{proof}
From the hypotheses we see that
\begin{eqnarray*}
\|q_0-q_1\| &\le &\|\pi(q_0-q_1)\| + \e L^\cC(q_0-q_1) \\
&\le &\|p_0-p_1\| + \e(L^\cC(q_0) + L^\cC(q_1)) = \d.
\end{eqnarray*}
Assume now that $\d < 1$.  Then according to Proposition~\ref{prop3.3} 
applied to $q_0$ and $q_1$ 
there is a path $t \to q_t$ from $q_0$ to 
$q_1$ with the stated properties.
\end{proof}

If $p_0 = p_1$ above then
we can obtain some additional information. The following proposition
is almost exactly proposition 4.3 of \cite{R17}. We do not
repeat the proof here.

\begin{prop}
\label{prop4.3}
With hypotheses as above, 
let $p \in \cA$, and let $q_0$ and $q_1$ be projections in
$\cC$ 
such that $\pi(q_0) = p = \pi(q_1)$.  If $\e L^\cC(q_0) < 1/2$ and 
$\e L^\cC(q_1) < 1/2$, then there is a path, $t \to q_t$, through
projections 
in $\cC$, from $q_0$ to $q_1$, such that $\pi(q_t) = p$ and 
\[
L^\cC(q_t) \le (1-\d)^{-1} \max\{L^\cC(q_0),L^\cC(q_1)\}
\]
for all $t$, where $\d = \e(L^\cC(q_0) + L^\cC(q_1))$.  
\end{prop}

By concatenating paths, 
we can combine the above results to obtain some information that does
not depend on $p_0$ and $p_1$ being close together. The next
proposition is almost exactly corollary 4.4 of \cite{R17}. 

\begin{cor}
\label{cor4.4}
Let $p_0$ and $p_1$ be projections in $\cA$, and let $q_0$ and 
$q_1$ be projections in $\cC$ such that $\pi(q_0) = p_0$ and 
$\pi(q_1) = p_1$.  Let $K$ be a constant such that $L^\cC(q_j) \le K$ for 
$j = 0,1$.  Assume further that there is a path $p$ from $p_0$ to 
$p_1$ such that for each $t$ there is a projection ${\tilde q}_t$ in 
$\cC$ such that $\pi({\tilde q}_t) = p_t$ 
and $L^\cC({\tilde q}_t) \le K$.  
Then for any $r > 1$ there is a continuous 
path $t \mapsto q_t$ of projections in 
$\cC$ going from $q_0$ to $q_1$ such that
\[
L^\cC(q_t) \le rK
\]
for each $t$.  (But we may not have $\pi(q_t) = p_t$ for all $t$.)  
\end{cor}

\begin{proof}
Given $r>1$, choose $\d > 0$ such that $(1-\d)^{-1} < r$,
and then choose an $\e > 0$ such that $2 \e K < \d$.
Then follow the proof of corollary 4.4 of \cite{R17} with
$N = K$. 
\end{proof}

Let us now see what consequences the above uniqueness results have 
when we have a bridge between two C*-algebras 
that are equipped with
suitable seminorms. 
Let $\cA$ and $\cB$ 
be unital C*-algebras and let
$\Pi = (\cD, \o)$ be a bridge from $\cA$ to $\cB$ . 
Let $L^\cA$ and $L^\cB$ be semi-finite lower-semi-continuous
strongly Leibniz slip-norms on $\cA$ and $\cB$. 
We use them to measure $\Pi$, 
and we assume that $r_\Pi$ is finite.
Let $\cC = \cA \oplus \cB$. 
Let $r \geq r_\Pi$ be chosen, and define
the seminorm $L^r$ on the C*-algebra $\cC$ by 
equation \ref{eqlr}.
Note that $L^r$ is admissible for $L^\cA$ and $L^\cB$
by Proposition \ref{admis}.
A projection in $\cC$ will now be of the form $(p,  q)$ where 
$p$ and $q$ are projections in $\cA$ and $\cB$ respectively.  
Roughly speaking, our main 
idea is that $p$ and $q$ will correspond if 
$L^r(p, q)$ is relatively small.  Notice that which projections 
then correspond to each other will strongly depend 
on the choice of $\Pi$ (just
as in \cite{R17}, where it was seen that which projections 
for ordinary compact metric spaces correspond
depends strongly on the choice of the metric that is put on the
disjoint union of the two metric spaces, as would be expected).  
We will only consider that 
projections correspond (for a given bridge $\Pi$) if there is some
uniqueness 
to the correspondence.  The following theorem gives appropriate 
expression for this uniqueness. It is our non-commutative
generalization of theorem 4.5 of \cite{R17}, and it
is an immediate consequence of Key Lemma \ref{keylem},   
Proposition~\ref{prop4.3}, and then Theorem~\ref{thm4.2}. 
The role of the $\e$ in Proposition~\ref{prop4.3} 
and Theorem~\ref{thm4.2} is now played by $h_\Pi + r$.  

\begin{thm}
\label{th4.5}
Let $\cA$ and $\cB$ be unital C*-algebras, and let 
$\Pi = (\cD, \o)$ be a bridge from $\cA$ to $\cB$.
Let $L^\cA$ and $L^\cB$ be 
lower semi-continuous strongly-Leibniz slip-norms 
on $\cA$ and $\cB$. Assume that the length of $\Pi$
as measured by $L^\cA$ and $L^\cB$ is finite. Let
$r \geq r_\Pi$ be chosen, and define $L^r$ on $\cC = \cA \oplus \cB$
by equation \ref{eqlr}.
\begin{itemize} 
\item[{\em a)}] Let $p \in \cA$ and $q \in \cB$ be
projections, 
and suppose that 
\[
(h_\Pi + r) L^r(p, q) < 1/2. 
\]
 If $q_1$ is another 
projection in $\cB$ such that $(h_\Pi + r) L^r(p, q_1) < 1/2$, then
there 
is a path $t \mapsto q_t$ through projections in $\cB$, going
from $q$ to $q_1$, such that
\[
L^r(p,q_t) \le (1-\d)^{-1} \max\{L^r(p, q),L^r(p, q_1)\}
\]
for all $t$, where $\d = (h_\Pi + r)(L^r(p, q) + L^r(p, q_1))$.  If
instead there is a 
$p_1 \in \cA$ such that $(h_\Pi + r) L^r(p_1, q) < 1/2$ then there is
a corresponding path from $p$ to $p_1$ with corresponding bound for 
$L^r(p_t, q)$.
\item[{\em b)}] Let $p_0$ and $p_1$ be projections in $\cA$ and
let 
$q_0$ and $q_1$ be projections in $\cB$.  Set
\[
\d = \|p_0-p_1\| + (h_\Pi + r)(L^r(p_0, q_0) + L^r(p_1, q_1)).
\]
If $\d < 1$ then there are continuous paths $t \mapsto p_t$ and 
$t \mapsto q_t$ from $p_0$ to $p_1$ and $q_0$ to $q_1$, respectively, 
through projections, such that
\[
L^r(p_t, q_t) \le (1-\d)^{-1} \max\{L^r(p_0, q_0), L^r(p_1, q_1)\}
\]
for all $t$.
\end{itemize}
\end{thm}

We remark that a more symmetric way of stating part b) above is to 
define $\d$ by
\[
\d = \max\{\|p_0-p_1\|,\|q_0-q_1\|\} + 
(h_\Pi + r)(L^r(p_0, q_0),L^r(p_1, q_1)).
\]

Let us now examine the consequences of Corollary~\ref{cor4.4}.  
This is best phrased in terms of:

\begin{notation}
\label{not4.6}
Let $\cP(\cA)$ denote the set of projections in
$\cA$.  
For any $s \in \bR^+$ let
\[
\cP^s(\cA) = \{p \in \cP(\cA): L^\cA(p) < s\},
\]
and similarly for $\cB$ and $\cC$.
\end{notation}

Now $\cP^s(\cA)$ may have many path components. 
As suggested by the main results
of \cite{R17}, it may well be appropriate, indeed necessary, to
view these different path components as 
representing \emph{inequivalent} vector
bundles, even if algebraically the vector bundles are isomorphic. 
That is the main idea of \cite{R17}, and of the present paper.
(Some additional perspective on this idea will be given
in Section \ref{secsums}.) 
Let $\S$ be one of these path
components.  
Let $\Phi_\cA$ denote the evident restriction map from $\cP(\cC)$ to $\cP(\cA)$ (for 
$\cC = \cA  \oplus  \cB$).  For a given $s' \in \bR^+$ with $s' \ge s$ it may be
that 
$\Phi_\cA(\cP^{s'}(\cC)) \cap \S$ is non-empty.  This is an existence
question, which we will not deal with here.  But at this point, from 
Corollary \ref{cor4.4} we obtain our non-commutative version of
theorem 4.7 of \cite{R17}, namely:

\begin{thm}
\label{th4.7}
Let notation be as above, and assume that $\mathrm{length}(\Pi) < \e$.  
Let $s \in \bR^+$ with $\e s < 1/2$.  Let $\S$ be a path component of 
$\cP^s(\cA)$.  Let $s' \in \bR^+$ with $s' \ge s$ and $\e s' < 1/2$.  Let 
$p_0, \ p_1\in \S$ and suppose that there are $q_0$ and $q_1$ in 
$\cP^{s'}(\cB)$ with $L^r(p_j , q_j) \le s'$ for $j = 0,1$.  Assume, 
even more, that there is a path $\tilde p$ in $\S$ connecting $p_0$
and 
$p_1$ that lies in $\Phi_\cA(\cP^{s'}(\cC))$.  Then for any $\d$ with 
$2\e s' < \d < 1$ there exist a path $t \mapsto p_t$ 
in $\cP(\cA)$ going from $p_0$ to $p_1$ 
and a path $t \mapsto q_t$ in $\cP(\cB)$ going from $q_0$ to $q_1$ 
such that $L^r(p_t, q_t) < (1-\d)^{-1}s'$ for each $t$.  The
situation 
is symmetric between $\cA$ and $\cB$, 
so the roles of $\cA$ and $\cB$ can be 
interchanged in the above statement.
\end{thm}  

Thus, in the situation described in the theorem, if $\S$  is
a connected path component of $\cP^s(\cA)$ that represents
some particular class of projective $\cA$-modules, 
then the projections $q \in \cP^s(\cB)$ 
paired with ones in $\S$ 
by the requirement that $L^r(p, q) < s$, will be homotopic, and in 
particular will determine isomorphic projective $\cB$-modules.  
We emphasize that 
the above pairing of projections depends strongly on 
the choice of $\Pi$, and not
just on the quantum Gromov-Hausdorff propinquity 
between $\cA$ and $\cB$.  
This reflects the fact that quantum Gromov--Hausdorff 
propinquity is only a metric on {\em isometry 
classes} of quantum compact metric spaces,
just as is the case for ordinary Gromov-Hausdorff
distance for ordinary compact metric spaces.

Notice that the homotopies obtained above 
between $q_0$ and $q_1$ need not 
lie in $\cP^s(\cB)$. We can only conclude that they lie in
$\cP^{s'}(\cB)$ 
where $s' = (1-\d)^{-1}s$.  But at least we can say that 
$s'$ approaches $s$ as $\e$, and so $\d$, goes to $0$.  

The results of this section suggest that the definition of
a ``C*-metric'' given in definition 4.1 of \cite{R21} should
be modified to use matrix seminorms, and so should
be given by:
\begin{defn}
Let $\cA$ be a unital C*-algebra. By a \emph{C*-metric}
on $\cA$ we mean a matrix Lip-norm (as defined in
Definition \ref{defmlip}), $\{L^\cA_n\}$, such that
each $L^\cA_n$ is strongly Leibniz. 
\end{defn}

We remark that in contrast to the contents of
section 6 of \cite{R17}, in the present paper
we do not include here any \emph{existence} theorems
for bundles on quantum spaces that are close
together. It appears that existence results
in the non-commutative case are more
difficult to obtain, but this matter remains
to be explored carefully.


\section{The algebras and the bridges}
\label{secalg}

In this section we will introduce the specific 
algebras and the bridges to which
we will apply the theory of the previous section. These are
described in \cite{R29} and in earlier papers on this topic,
but in greater generality than we use in the later parts of
the present paper. Nevertheless, here we will begin by 
reviewing this more general setting, since
it gives useful context, and the main results of this paper should
eventually be generalized to the more general setting. 

Let $G$ be a compact group (perhaps even finite at first, but later to be
$SU(2)$).  Let $U$ be an
irreducible unitary representation of $G$ on a 
(finite-dimensional) Hilbert space ${\mathcal
H}$.  Let $\cB = \cL({\mathcal H})$ denote the $C^*$-algebra of all linear
operators on ${\cH}$ (a ``full matrix algebra'', with its operator
norm).  There is a natural action, $\a$, of $G$ on $\cB$ by conjugation by
$U$, that is, $\a_x(T) = U_xTU_x^*$ for $x \in G$ and $T \in \cB$.  Because
$U$ is irreducible, the action $\a$ is ``ergodic'', in the sense that the
only $\a$-invariant elements of $\cB$ are the scalar multiples of the
identity operator. 


 Fix a continuous length function, $\ell$, on $G$ (so
$G$ must be metrizable).  Thus $\ell$ is non-negative, $\ell(x) = 0$ iff
$x = e_G$ (the identity element of $G$), $\ell(x^{-1}) = \ell(x)$, and
$\ell(xy) \le \ell(x) + \ell(y)$. We also require 
that $\ell(xyx^{-1}) = \ell(y)$ for all $x$, $y \in G$.  
Then in
terms of $\a$ and $\ell$ we can define a 
seminorm, $L^\cB$, on $\cB$ by the formula
\begin{equation}
\label{lipn}
L^\cB(T) = \sup\{ \|\a_x(T) - T \|/\ell(x): x \in G \quad \mathrm{and} 
\quad x \neq e_G\}  .
\end{equation}
Then $(\cB,L_\cB)$ 
is an example of a compact C*-metric-space, 
as defined in definition 4.1 of \cite{R21}.
In particular, $L_\cB$ satisfies the conditions
given there for being a Lip-norm, recalled 
in Definition \ref{deflip} above. 

Let $P$ be a rank-1 projection in $\cB$ (soon to be
the projection on a highest weight subspace). Let $H$
be the stability subgroup of $P$ for $\a$. Form the
quotient space $G/H$ (which later will be the sphere).
We let $\l$ denote the action
of $G$ on $G/H$, and so on $\cA = C(G/H)$, by left-translation. 
Then from $\l$ and $\ell$ we likewise obtain a seminorm, 
$L^\cA$, on $\cA$ by the evident analog of
formula \ref{lipn}, except that we must now permit $L^\cA$
to take the value $\infty$. It is shown in proposition 2.2
of \cite{R4} that the set of functions for which
$L^\cA$ is finite (the Lipschitz functions) is a 
dense $*$-subalgebra of $\cA$.  Also, $L^\cA$ 
is the restriction to $\cA$ of the seminorm on
$C(G)$ that we get from $\ell$ and left translation,
when we view $C(G/H)$ as a subalgebra of
$C(G)$, as we will do when convenient.   
From $L^\cA$ we can use equation \ref{metr} 
to recover the usual quotient metric \cite{Wvr2}
on $G/H$ coming from the metric on $G$ determined by $\ell$.  
One can check easily that $L^\cA$ in turn comes from
this quotient metric.  Thus $(\cA,L^\cA)$ 
is the compact C*-metric-space
associated to this ordinary compact metric space. 
Then for any bridge from $\cA$ to $\cB$ 
we can use $L^\cA$ and $L^\cB$ to measure the
length of the bridge in the way given by {\Lat} \cite{Ltr2}, 
which we described in Definitions \ref{reach} and \ref{height}. 

We now describe the natural bridge,  
$\Pi=(\cD, \o)$, from $\cA$ to $\cB$ that was
first presented in section 2 of \cite{R29}. 
We take $\cD$ to be the C*-algebra 
\[
\cD = \cA \otimes \cB = C(G/H, \cB)  .
\]
We identify $\cA$ with the 
subalgebra $\cA \otimes 1_\cB$ of $\cD$,
where $1_\cB$ is the identity element of $\cB$. 
Similarly, we identify $\cB$ with the 
subalgebra $1_\cA \otimes \cB$ of $\cD$.
In view of many of the calculations done in \cite{R7, R21} it is
not a surprise that we define the pivot $\o$ to be the function in  
$C(G/H, \cB)$ defined by
\[
\o(x) = \a_x(P)
\]
for all $x \in G/H$, where $P$ is the rank-1 projection
chosen above. (It is a ``coherent state''.) 
We notice that $\o$ is actually 
a non-zero projection
in $\cD$, and so it satisfies the requirements for being a pivot.

But projective modules over algebras are in general given by
projections in matrix algebras over the given algebra, not just
by projections in the algebra itself. This brings us back 
to the topic of matricial
bridges which was introduced early in Section \ref{secvbclose}. 
We now apply the general matricial framework
discussed there to the more specific  
situation described just above in which $\cA = C(G/H)$, etc.,
with
corresponding natural bridge $\Pi$, and then with its
associated 
matricial bridges $\Pi_d$ defined
as in Definition \ref{defmat}. We must
specify our matrix slip-norms. This is essentially done
in example 3.2 of \cite{Wuw2} and section 14 of \cite{R21}.
Specifically:

\begin{notation}
\label{matbrid}
As above, we have the 
actions $\l$ and $\a$ on $\cA = C(G/H)$ and $\cB = \cB(\cH)$
respectively. For any natural number $d$ 
let $\l^d$ and $\a^d$ be the
corresponding actions $\iota_d\otimes \l$ and $\iota_d \otimes \a$
on $M_d\otimes \cA = M_d(\cA)$ and $M_d\otimes \cB = M_d(\cB)$,
for $\iota_d$ denoting the identity operator from $M_d$ to itself.
We then use the length function $\ell$ and formula  \ref{lipn} to
define seminorms $L^\cA_d$ and $L^\cB_d$ on $M_d(\cA)$
and $M_d(\cB)$. 
\end{notation}

\noindent
It is easily verified that $\{L^\cA_d\}$ and
$\{L^\cB_d\}$ are matrix slip-norms. Notice that here $L_1^\cA = L^\cA$
and $L_1^\cB = L^\cB$ are actually Lip-norms, and so, by property
1 of Definition \ref{defmtx}, for each $d$ the null-spaces of $L^\cA_d$ and
$L^\cB_d$ are exactly $M_d$.

We remark that, as discussed in \cite{R29}, the bridge $\Pi = (\cD, \o)$
with $\cD = C(G/H, \cB)$ considered above is an example of a
``bridge with conditional expectations'', and that for such bridges
theorem 5.5 of \cite{R29} gives upper bounds for the reach and
height of $\Pi_d$ in terms of the choices of $\ell$, $P$, etc.
In particular, they are finite.


\section{Projections for $\cA$}
\label{secprojA}

We now restrict our attention to the case in which $G = SU(2)$. We
choose our notation in such a way that much of 
it generalizes conveniently to the
setting of general compact semi-simple Lie groups, though we do not
discuss that general case here. We let $H$ denote the diagonal subgroup
of $G$, which is a maximal torus in $G$.  The homogeneous 
space $G/H$ is diffeomorphic to the
2-sphere. As before, we set $\cA = C(G/H)$. 

It is known that for any 2-dimensional
compact space every complex vector bundle is a direct sum
of complex line bundles. See theorem 1.2 of 
chapter 8 of \cite{Hsm}. This applies to the 2-sphere, and
so in this and the next few sections we will concentrate
on the case of line bundles. In Section \ref{secsums} we
will discuss the situation for direct sums of projective
modules. It is not entirely straight-forward.

In this section we seek formulas for projections that represent 
the line bundles
over $G/H$. We will follow the approach given in
\cite{Lnd2}, where formulas for the projections were first given
in the global form that we need. (See also \cite{Lnd1, LnS, R22}.) 
But the formulas
given in \cite{Lnd2} do not seem convenient for obtaining the
detailed estimates that we need later, 
so the specific path that we follow is
somewhat different. 

We will often view (i.e. parametrize) $H$ as $\bR/\bZ$. We define the
function $e$ on $\bR$, and so on $H$, 
by $e(t) = e^{2 \pi i t}$. Then each irreducible
representation of $H$ is of the form $t \mapsto e(kt)$ for some $k \in \bZ$.
For each $k\in \bZ$ let $\Xi_k$ denote the corresponding $\cA$-module
defined by:
\begin{notation}
\label{amod}
\[
\Xi_k = \{\xi \in C(G, \bC): \xi(xs) = \bar e(ks) \xi(x) \ \ \mathrm{for \ all}  \ \ 
x \in G, s \in H\},
\]
\end{notation}
\noindent
where elements of $\cA$ are viewed as functions on $G$ that act 
on $\Xi_k$ by pointwise multiplication. Then $\Xi_k$
is the module of continuous cross-sections of a fairly evident
vector bundle (a complex line bundle) over $G/H$. For $k \neq 0$
these are the physicists' ``monopole bundles''. 
Their ``topological charge'', or first Chern number, is
$k$ (or $-k$ depending on the conventions used). See
sections 3.2.1 and 3.2.2 of \cite{Lnd2}.
We let $\l$ denote the action of $G$ on
$\cA$, and also on $\Xi_k$, by left translation. These actions are
compatible, so that $\Xi_k$ is a $G$-equivariant $\cA$-module,
reflecting the fact that the
corresponding vector bundle is $G$-equivariant.

In order to apply the theory of Section \ref{secproleib} we need to
find a suitable projection from a free $\cA$-module onto $\Xi_k$. 
We do this in the way discussed in section 13 of \cite{R17}.
The feature that we use
to obtain the projections is the well-known fact that the one-dimensional
representations of $H$ occur as sub-representations of the restrictions
to $H$ of finite-dimensional unitary representations of $G$. Since $H$ 
is a maximal torus in $SU(2)$, the integers determining the one-dimensional
representations which occur when restricting a representation of $G$ are,
by definition, the weights of that representation.
We recall \cite{Smn} that for each non-negative integer $m$ there is an irreducible
representation, $(\cH^m, U^m)$ of $G$ whose weights are 
$m, m-2, \dots, -m + 2, -m$, each of multiplicity 1, 
and such a representation is unique up
to unitary equivalence. In particular, the dimension of $\cH^m$ is $m+1$.
The integer $m$ is called the ``highest weight''
of the representation. 
For a given integer $k$ (which may be negative) that
determines the $\cA$-module $\Xi_k$, 
we choose to consider the representation $(\cH^{|k|}, U^{|k|})$.

Then the one-dimensional subspace $\cK$ of $\cH^{|k|}$ for the 
highest weight if $k$ is
non-negative, or for the lowest weight if $k$ is negative,
is carried into itself by the restriction of $U^{|k|}$ to 
the subgroup $H$, and 
this restricted representation of $H$ is equivalent to the one-dimensional
representation of $H$ determining $\Xi_k$.  From now
on we simply let $V$ denote this restricted representation of $H$ on $\cK$.
Set
\[
\Xi^V_k = \{\xi \in C(G,\cK): \xi(xs) = V_s^*(\xi(x)) \mbox{ for } x \in
G,\ s \in H\}.
\]
Clearly $\Xi^V_k$ is a module over $\cA = C(G/H)$ 
that is isomorphic to $\Xi_k$. 

We want to show that $\Xi^V_k$ is a projective $A$-module, and 
to find a projection representing it.  Set 
\[
\U_k = C(G/H,\cH^{|k|})   .
\]
Then any choice of basis for $\cH^{|k|}$ exhibits
$\U_k$ as a free $A$-module.  For 
$\xi \in \Xi^V_k$ set $(\Phi\xi)(x) = U^{|k|}_x\xi(x)$ for $x \in G$, and
notice that $(\Phi\xi)(xs) = (\Phi\xi)(x)$ for $s \in H$ and 
$x \in G$, so that $\Phi\xi \in \U_k$.  It is clear that $\Phi$ is
an injective $A$-module homomorphism from $\Xi_V$ into $\U_k$.  
We show that the range of $\Phi$ is projective by exhibiting the
projection onto it from $\U_k$. This projection is the one that 
we will use in the later sections to represent the projective
module $\Xi_k$.  
\begin{notation}
\label{bigP}
We denote the projection from $\cH^{|k|}$ onto $\cK$ by $P^k$.
\end{notation}
Note that
$U^{|k|}_s P^k U^{|k|*}_s = P^k$ for $s \in H$ 
by the $H$-invariance of $\cK$.  Let $\cE_k$
denote the $C^*$-algebra $C(G/H,\cL(\cH^{|k|}))$. 
In the evident way 
$\cE_k = \mathrm{End}_A(\U_k)$.  Define $p_k$ on $G$ by 
\begin{equation}
\label{prj1}
p_k(x) = U^{|k|}_xP^kU_x^{|k|*},
\end{equation}
and notice that $p_k(xs) = p_k(x)$ for $s \in H$ and $x \in G$, so that 
$p_k \in \cE_k$.  Clearly $p_k$ is a projection 
in $\cE_k = \mathrm{End}_A(\U_k)$.  

\begin{prop}
\label{repprj}
As an operator on $\U_k$, the range of the projection $p_k$
is exactly the range of the injection $\Phi$.
\end{prop}
\begin{proof}
If $\xi \in \Xi_V$, then 
$p_k(x)(\Phi\xi)(x) = U^{|k|}_x P^k U^{|k|*}_x U^{|k|}_x \xi(x) 
= (\Phi\xi)(x)$, so that
$\Phi\xi$ is in the range of $p_k$.  Suppose, conversely, that 
$F \in \U_k$ and that $F$ is in the range of $p_k$.  Set 
$\eta_F(x) = U_x^{|k|*} F(x) = U^{|k|*}_xp_k(x)F(x) = P^k U^{|k|*}_xF(x)$.  
Then the range
of $\eta_F$ is in $\cK$, and we see easily that 
$\eta_F(xs) = U_s^{|k|*}\eta_F(x)$.  Thus $\eta_F \in \Xi_V$.  Furthermore,
$(\Phi\eta_F)(x) = F(x)$.  Thus $F$ is in the range of $\Phi$.  This
shows that the range of $p_k$ as a projection on $\U_k$ is exactly
the range of $\Phi$.  
\end{proof}
Thus the range of $\Phi$, and so also $\Xi^V_k$, are projective
$A$-modules that are isomorphic, 
and $p_k$ is a projection that represents $\Xi^V_k$, 
and so represents $\Xi_k$.
It is this
projection $p_k$ that we will use in the later parts of this paper.

To express $p_k$ as an element of $M_d(A)$ for $d = |k| + 1$ 
we need only
choose an orthonormal basis, $\{e_j\}_{j=1}^d$, for $\cH^{|k|}$, 
and view the corresponding constant functions
as a basis (so standard module frame) for $\U_k$, and then 
express $p_k$
in terms of this basis.  Furthermore, if we define $g_j$ on $G$ by
$g_j(x) = P^k U_x^{|k|*}e_j$, then it is easily seen that each $g_j$ is
in $\Xi_V$, and that $\{g_j\}$ is a standard module frame,
as defined in definition 7.1 of \cite{R17},
for $\Xi_V$.  The basis also gives us an isomorphism of $\cE_k$ with
$M_d(A)$. But it is more natural and convenient to
view $p_k$ as an element of $\cE_k = \mathrm{End}_A(\U_k)$.
To summarize:

\begin{notation}
For $p_k$ defined as in equation \ref{prj1}, we 
use the identification of $M_d(\cA)$ with $\cE_k =C(G/H, \cL(\Hk))$
to view $p_k$ as 
an element of $M_d(\cA)$,    
and we use $p_k$ as the projection representing the projective $\cA$-module
$\Xi_k$.
\end{notation}


\section{Projections for $\cB^n$}
\label{secprobn}

Let $(\Hn, U^n)$ be the irreducible representation of $G = SU(2)$
of highest weight $n$. Let $\Bn = \cL(\Hn)$, and let $\a$ be the
action of $G$ on $\Bn$ by conjugation, that is, 
$\a_x(T) = U^n_x T U^{n*}_x$ for $T \in \Bn$ and $x \in G$.

Suitable projective modules for our context seem to have
been first suggested
in \cite{GKP}. (See the paragraph after equation 40 there.) 
See also equation 6.6 of \cite{HWK1}.
The formulation closest to that which we use here is found in
equation 84 of \cite{CSW}.
For each $k\in \bZ$ let $\Onk$ denote the right $\Bn$-module
defined by:
\begin{notation}
\label{bmod}
\[
\Onk = \cL(\Hn, \Hnk) ,
\]
where $(\Hnk, \Unk)$ is the irreducible representation
of highest weight $k+n$.
Thus if $k < 0$ we need $n$ large enough that $k+n \geq 0$.
\end{notation}
\noindent
Then $\Onk$ is a right $\Bn$-module by composing operators
in $\Onk$ on the right by operators in $\Bn$. 

We want to embed $\Onk$ into a free $\Bn$-module so that we
can consider the corresponding projection. Let 
\[
\U^n_k = \cL(\Hn, \cH^{|k|} \otimes \Hn)
\]
with its evident right action of $\Bn$ by composing operators. 
Then $\U^n_k$
is naturally isomorphic to $\cH^{|k|} \otimes \Bn$, so that
it is indeed a free $\Bn$-module, of rank the dimension
of $\cH^{|k|}$, which is $d = |k|+1$.

If $k>0$ and if $\eta^k$ and $\eta^n$ are highest weight vectors in 
$\cH^k$ and $\Hn$, then $\eta^k \otimes \eta^n$ is a highest
weight vector of weight $k+n$ in $\cH^k \otimes \Hn$,
for the action $U^k \otimes U^n$,  and thus $\cH^k \otimes \Hn$
contains a (unique) copy of $\Hnk$. If $k < 0$ but $k+n\geq 0$
then $\cH^{|k|} \otimes \Hn$ again contains a highest weight
vector of weight $k+n$, but the argument is somewhat more
complicated, and we give it in Lemma \ref{highwt}. Thus
again $\cH^{|k|} \otimes \Hn$ contains a (unique)
copy of $\Hnk$.
Consequently, for any $k$ we can,
and do, identify $\Hnk$ with the corresponding subspace of
$\cH^{|k|} \otimes \Hn$. (We always assume that $k+n \geq 0$.)
Accordingly, we identify $\Onk$ with a $\Bn$-submodule
of $\U^n_k$. 


We have an evident left action of $\cL(\cH^{|k|} \otimes \Hn)$
on $\U^n_k$ by composing operators in $\U^n_k$ on the left
by operators in $\cL(\cH^{|k|} \otimes \Hn)$. In fact, $\U^n_k$ is a 
$\cL(\cH^{|k|} \otimes \Hn)$-$\cL(\Hn)$-bimodule, and 
because $\cL(\Hn)=\Bn$ there is an evident
natural isomorphism
\[
End_\Bn(\U^n_k) \cong \cL(\cH^{|k|} \otimes \Hn)   .
\]
(The bimodule $\U^n_k$ gives a
Morita equivalence between $\cL(\cH^{|k|} \otimes \Hn)$
and $\Bn$.)
But we also have natural isomorphisms
\[
\cL(\cH^{|k|} \otimes \Hn) \cong \cL(\cH^{|k|}) \otimes \cL(\Hn)
\cong M_d(\cB^n),
\]
and we will use these to take $\cL(\cH^{|k|} \otimes \Hn)$                   
as our version of $M_d(\cB^n)$.

Let $p^n_k$ denote the projection of  $\cH^{|k|} \otimes \Hn$
onto its subspace  $\Hnk$, and view $p^n_k$ as an element
of $\cL(\cH^{|k|} \otimes \Hn)$.
Then composition with $p^n_k$ on the left 
gives an evident projection
of $\cL(\Hn,\cH^{|k|} \otimes \Hn)$ onto its submodule
$\cL(\Hn, \Hnk)$, that is, from $\U^n_k$ onto $\Onk$, 
respecting the right action of $\Bn$. We thus see that
the projection $p^n_k$ is a projection
that represents the projective
$\Bn$-module $\Onk$. 

\begin{notation}
\label{repproj}
Let $p^n_k$ denote the projection of  $\cH^{|k|} \otimes \Hn$
onto its subspace  $\Hnk$. We use the identification of
$M_d(\cB^n)$ with $\cL(\Hk \otimes \Hn)$ to view $p^n_k$
as an element of $M_d(\cB^n)$, and we use $p^n_k$
as the projection to represent the projective
$\Bn$-module $\Onk$. 

\end{notation}

Suppose now that, as in Section \ref{secalg}, 
we have chosen a continuous
length-function $\ell$ on $G$ which we then use to define 
slip-norms for various actions. For the proof of our main
theorem (Theorem \ref{mainthm}), concerning the convergence
of modules, we need a bound
for $L^{M_d(\Bn)}(p^n_k)$ that is independent of $n$, 
where $d = |k| + 1$. That is, we need:
\begin{prop}
\label{pknbound}
For fixed $k$ there is a constant  $c_k$ (depending
in particular on the choice of the length function 
$\ell$) such that
\[
L^{M_d(\Bn)}(p^n_k) \leq c_k
\]
for all $n$.
\end{prop}
\begin{proof}
Notice that because the representation of $G$
on $\Hnk$ is a subrepresentation of the representation
$U^{|k|} \otimes U^n$
on $\cH^{|k|} \otimes \Hn$, the operator $p^n_k$ is
invariant under the corresponding conjugation action of $G$
on $\cL(\cH^{|k|} \otimes \Hn)$. 
But the action of $G$ used to define $L^{M_d(\Bn)}$, as
discussed in Section \ref{secalg},  comes from the action
$\a$ of $G$ on $\Bn$ using the representation $U^n$ on $\Hn$,
and is the action $\b^n = \iota_d \otimes \a$ coming from
conjugating elements of $\cL(\cH^{|k|}  \otimes \Hn)$ by
the representation $I_d \otimes U^n$ 
on $\cH^{|k|} \otimes \Hn$. Thus we must consider 
$\b^n_x(p^n_k)$, and the action $\b^n$ depends strongly
on $n$.

Now 
\[
\b^n_x(p^n_k) = (I_d \otimes U^n_x)p^n_k(I_d \otimes U^{n*}_x)  .
\]
But as said above, $p^n_k$ is invariant under conjugation by
$U^{|k|} \otimes U^n$, so we can replace $p^n_k$ in the
above equation by
\[
(U_x^{|k|*} \otimes U_x^{n*})p^n_k(U_x^{|k|} \otimes U_x^n),
\]
from which we find that
\[
\b^n_x(p^n_k) =  (U_x^{|k|*} \otimes I^n)p^n_k(U_x^{|k|} \otimes I^n), 
\]
where $I^n$ is the identity operator on $\Hn$.
We can express this as
\[
\b^n_x(p^n_k) = (\a^k_{x^{-1}} \otimes \iota^n)(p^n_k)  
\]
where $\a^k$ is the conjugation action of $G$ on $\cL(\cH^{|k|})$
and $\iota^n$ is the identity operator on $\cL(\Hn)$ 
(but this does not work for most other operators 
on $\cL(\cH^{|k|}  \otimes \Hn)$ besides $p^n_k$.)
Let $\g^k$ denote the action on $\cL(\cH^{|k|}  \otimes \Hn)$ 
defined by
$\g^k_x = (\a^k_x \otimes \iota^n)$. Then we see that we
have obtained
\[
\b^n_x(p^n_k) = \g^k_{x^{-1}}(p^n_k)    ,
\]
(and the action $\g^k$ depends only very weakly on $n$).
It follows that 
\[
(\b^n_x(p^n_k) - p^n_k)/\ell(x) = (\g^k_{x^{-1}}(p^n_k) - p^n_k)/\ell(x^{-1} ) ,
\]
where we have used that $\ell(x) = \ell(x^{-1})$. From this it follows
that
\[
L^{\b^n}(p^n_k) = L^{\g^k}(p^n_k)  .
\]
Consequently, because $\|p^n_k\| = 1$ for all $n$, 
the following lemma will conclude
the proof.

\begin{lem}
\label{lemindy}
For any $k$ there is a constant, $c_k$, such that for
any $n$ 
we have
\[
L^{\g^k}(T) \leq c_k \|T\|
\]
for every $T \in \cL(\cH^{|k|}  \otimes \Hn)$.  
\end{lem}

\begin{proof}
We use a standard ``smoothing''-type argument.
Let $f \in C(G)$, and let $\g^k_f$ be the
integrated form of $\g^k$ applied to $f$. Then for any
$x \in G$ and $T \in  \cL(\cH^{|k|}  \otimes \Hn)$ we have
\begin{align*}
\g^k_x(\g^k_f(T)) &= \g^k_x(\int f(y)\g^k_y(T)dy)
=\int f(y)\g^k_{xy}(T)dy    \\
&= \int f(x^{-1}y)\g^k_y(T)dy = \int (\l_x(f))(y)\g^k_y(T)dy,
\end{align*}
where $\l$ is the action of left-translation on $C(G)$.
Now suppose further that $L^\l(f) < \infty$, where $L^\l$
is the Lipschitz seminorm on $C(G)$  using formula \eqref{lipn},
for the length function $\ell$ and the action $\l$. 
Then
\begin{align*}
(\g^k_x(\g^k_f(T)) - \g^k_f(T))/ \ell(x)
= \int ((\l_x f)(y) - f(y))/\ell(x)) \g^k_y(T) dy. 
\end{align*}
On taking norms and then supremum over $x \in G$, we find that
\begin{equation}
\label{eqsmth}
L^{M_d(\Bn)}(\g^k_f(T)) \leq L^\l(f)\|T\|.
\end{equation}

It is shown in proposition 2.2 of \cite{R4} that the collection
of functions $f$ for which $L^\l(f) < \infty$ is a norm-dense
$*$-subalgebra of $C(G)$. Thus this sub-algebra
will contain an approximate identity for the convolution
algebra $L^1(G)$. As we let $f$ run through such
an approximate identity, $\a^k_f$ will converge for
the strong operator topology to the identity operator
on $\cL(\cH^{|k|})$. But $\cL(\cH^{|k|})$ is finite-dimensional,
and so the convergence is also for the operator norm. Thus we
can find $f$ such that $\a^k_f$ is close enough to the
identity operator that $\a^k_f$ is invertible, which
in turn implies that $\g^k_f = \a^k_f \otimes \iota^n$ 
is invertible, and that $\|(\g^k_f)^{-1}\| = \|(\a^k_f)^{-1}\|$. 
For such a fixed $f$ we have, on using inequality \ref{eqsmth},
\begin{align*}
L^{M_d(\Bn)}(T) &= 
L^{M_d(\Bn)}(\g^k_f((\g^k_f)^{-1}(T))) \\
&\leq L^\l(f)  \|(\g^k_f)^{-1}(T)\|  
\leq L^\l(f)\|(\g^k_f)^{-1}\|\|T\|.
\end{align*}

Thus, with notation as just above, we can set
\[
c^k = L^\l(f)\|(\g^k_f)^{-1}\|\
\]
\end{proof}   
This concludes the proof of Proposition \ref{pknbound}.
\end{proof}   

\section{Bridges and projections}
\label{secbp}

In this section we begin to
apply the general results about bridges and projections
given in Section \ref{secproleib}
to the specific algebras, projections and
bridges for $G = SU(2)$ described 
in Sections \ref{secalg}, \ref{secprojA}, and \ref{secprobn} .
We fix the positive integer $n$ and the integer $k$, to be
used as in the sections above, and we set $d = |k| + 1$. 
We let $\cD^n = \cA \otimes \cB^n$, and we let 
$\Pi^n_d = (M_d(\cD^n), \o_d^n)$ for $\o_d^n$ 
defined in Definition \ref{defmat}, so that
$\Pi^n_d$ is a bridge from $M_d(\cA)$ to $M_d(\cB^n)$.
We let $p_k \in M_d(\cA)$ and 
$p^n_k \in M_d(\Bn)$ be the projections defined in the previous
two sections, and we view them as elements of 
$M_d(\cD^n)$ via
the injections of $M_d(\cA)$ and 
$M_d(\Bn)$ into $M_d(\cD^n)$ given earlier. 
For this purpose we use the identification of 
$M_d(\cA)$ with $C(G/H, \cL(\Hk))$, of
$M_d(\cB^n)$ with $\cL(\Hk \otimes \Hn)$,
and the identification of 
$M_d(\cD^n)$
with $C(G/H, M_d \otimes \cB^n) = C(G/H, \cL(\Hk \otimes \Hn))$.
Thus:
\begin{notation}
\label{notproj} When viewed as elements 
of $C(G/H, \cL(\Hk \otimes \Hn))$, 
the projection $p_k$ is defined by
\[
 p_k(x) = \Ukx P^k\Ukx{^*} \otimes I^n
 \] for $x \in G/H$, while the projection
$p^n_k$ is defined as the constant function 
\[
p^n_k(x) = p^n_k
\] 
on $G/H$.
\end{notation}
Then, as discussed
in \cite{R29} and in Section \ref{secproleib}, we need to obtain a
useful bound for 
\[
\|p_k\o_d^n - \o_d^n p^n_k\|.
\]
Now
\begin{align*}
p_k(x)\o_d^n(x) &= (\Ukx P^k\Ukx{^*} \otimes I^n)
(I_d \otimes U^n_x P^n U^{n*}_x)   \\
&= (\Ukx \otimes U^n_x )(P^k \otimes P^n)(\Ukx{^*} \otimes U^n_x{^*} )  ,
\end{align*}
while  
\[
\o_d^n(x)p^n_k(x) = (I_d \otimes U^n_x P^n U^{n*}_x) p^n_k   .
\]
But the subspace $\Hnk$ of $\Hk \otimes \Hn$ is carried into itself
by the representation $ U^{|k|} \otimes U^n$, and so
\[
p^n_k =  (U_x^{|k|} \otimes U_x^n) p^n_k (U_x^{|k|*} \otimes U_x^{n*}),
\]
so that
\begin{align*}
(I_d \otimes &U^n_x P^n U^{n*}_x) p^n_k   \\  
&= (I_d \otimes U^n_x P^n U^{n*}_x) 
(U^{|k|}_x \otimes U^n_x)p^n_k (U^{|k|*}_x \otimes U^{n*}_x)   \\
&= (U^{|k|}_x \otimes U^n_x)(I_d \otimes P^n)p^n_k(U^{|k|*}_x \otimes U^{n*}_x)
\end{align*}
Thus
\begin{align*}
p_k(x)\o_d^n(x)  &- \o_d^n(x)p^n_k(x)   \\  
 &= (U^{|k|}_x \otimes U^n_x)
 (P^k \otimes P^n  -  (I_d \otimes P^n)p^n_k)
 (U^{|k|*}_x \otimes U^{n*}_x)   ,
\end{align*} 
and consequently
\begin{equation*}
\|p_k(x)\o_d^n(x)  - \o_d^n(x)p^n_k(x) \| = 
\|(P^k \otimes P^n)  -  (I_d \otimes P^n)p^n_k\|,
\end{equation*}
which is independent of $x$. Thus
\begin{equation}
\label{eqpro}
\|p_k\o_d^n  - \o_d^n p^n_k \| = 
\|(P^k \otimes P^n)  -  (I_d \otimes P^n)p^n_k\|  .
\end{equation}
 The next two sections are 
devoted to obtaining suitable upper bounds for
the term on the right.


\section{The core calculation for the case of $k \geq 0$}
\label{core}

We treat first the case in which $k \geq 1$. The case in
which $k \leq -1$ is somewhat more complicated, and
we treat it in the next section. (The case for $k=0$
is trivial.)
Fix $k \geq 1$. Let $T^n_k$ be the negative of the 
operator whose norm
is taken on the right side of equation \ref{eqpro}. 
Notice that  $P^k \otimes P^n = (P^k \otimes P^n)p^n_k $ 
(which is false for $k \leq -1$),
and that $P^k \otimes P^n$ commutes with $p^n_k$. 
Consequently
\[
T^n_k = ((I_d   -  P^k) \otimes P^n)p^n_k.
\]
It is an operator on $\cH^k \otimes \Hn$. 
To understand the structure of this operator we use the
weight vectors of the two representations involved. For this
purpose we use the ladder operators in the complexified Lie algebra
of $SU(2)$. There are many conventions for them used
in the literature. Since the calculations in this section are
crucial for our main results, we give in Appendix 1
a careful statement of the conventions we use, and of the
consequences of our conventions. As explained in more detail there,
we let $H$ span the Lie subalgebra of our maximal torus, and
we let $E$ and $F$ be the ladder operators that satisfy the
relations 
\[
[E,F] = H, \quad \quad [H, E] = 2E, \quad \quad 
\mathrm{and} \quad [H, F] = -2F  .
\]
For each $n$ these elements of the complexified Lie algebra 
of $SU(2)$  act on $\cH^n$ via the infinitesimal
form of $U^n$, but as is commonly
done we will not explicitly 
include $U^n$ in our formulas. 
As operators on $\cH^n$
they satisfy the
relation $E^* = F$.
Let $f_n$  be a 
highest-weight vector for
the representation $(\cH^n, U^n)$, with $\|f_n\| = 1$.
The weights of this representation are $n, n-2, n-4, \dots, -n+2, -n$.
Set
\[
f_{n-2a} = F^a f_n
\]
for $a=0, 1, 2,  \dots, n$. These vectors form an orthogonal
basis for
$\cH^n$. As shown the Appendix, we have
\[
\| f_{n-2a}\|^2 = \|F^a f_n\|^2 = a!\Pi _{b=0}^{a-1} (n-b)  .
\]
for $a=0, 1, 2,  \dots, n$.

 Much as above,
let $e_k$  be a 
highest-weight vector for
the representation $(\cH^k, U^k)$ , with $\|e_k\| = 1$.
Set
\[
e_{k-2a} = F^a e_k
\]
for $a=0, 1, 2,  \dots, k$. These vectors form an orthogonal
basis for
$\cH^k$.

Set $v_{k+n} = e_k\otimes f_n$. It is a highest weight
vector of weight $k+n$ in $\cH^k \otimes \cH^n$ 
for the representation
$U^k \otimes U^n$. Notice that $\|v_{k+n} \| = 1$. Then set
\[
v_{k+n-2a} = F^a v_{k+n}
\] 
for $a=0, 1, 2,  \dots, k+n$. These vectors form an
orthogonal basis for a sub-representation of
$(\cH^k \otimes \cH^n, U^k \otimes U^n)$ that is
unitarily equivalent to the irreducible representation
$(\cH^{k+n}, U^{k+n})$, and we will identify it with 
the latter. 
The span of these vectors is the range of
the projection $p^n_k$, and so from the form of 
$T^n_k$ we see that we only need to calculate the norm of
$T^n_k$ on the span of these vectors.
For a given $a$ we have
\[
F^a(e_k \otimes f_n) = (F^a e_k)\otimes f_n \
+ \ \mathrm{lower \ order \ terms}   ,
\]
where the lower-order terms are of the form
$F^{a-b}e_k\otimes F^b f_n$ for some integer $b\geq 1$.
But each of these lower-order terms is in the kernel of
$(I_d -P^k)\otimes P^n$ because $P^n(F^b f_n) = 0$
for $b \geq 1$. The highest weight vector
$e_k \otimes f_n$ is also in that kernel, because
$P^k$ is the projection onto the span of $e_k$. Thus we find 
that $T^n_k(e_k \otimes f_n) = 0$, while
\[
T^n_k(F^a(e_k \otimes f_n)) = (F^a e_k) \otimes f_n
\]
for $a = 1, \dots, k$. But the terms $(F^a e_k) \otimes f_n$
are orthogonal to each other for different $a$'s. Because
$\|(F^a e_k) \otimes f_n\| = \|F^a e_k\|$ for each $a$,
it follows that
\[
\|T^n_k\| = \max\{\|F^a e_k\|/ \|F^a(e_k \otimes f_n)\|:a = 1, \dots, k\}.
\]
But from equation \eqref{keyeq} of the Appendix we see that for each
$a$ we have
\[
\|F^a e_k\|^2 = a!\Pi_{b=0}^{a-1} (k-b),
\]
while
\[
\|F^a(e_k \otimes f_n)\|^2 = a!\Pi_{b=0}^{a-1} (k+n-b).
\]
Thus
\[
\|F^a e_k\|^2/ \|F^a(e_k \otimes f_n)\|^2
= \Pi_{b=0}^{a-1} (k-b)/(k+n-b)  .
\]
Since $(k-b)/(k+n-b) <1$ for each $b = 1, \dots, k$,
it is clear that the maximum of these products depending
on $a$ occurs when $a=1$ and so we find that 
\[
\|T^n_k\| = (k/(k+n))^{1/2}.
\]
We thus obtain:
\begin{prop}
\label{propcore1}
In terms of our notation in section \ref{secbp}, 
and by equation \ref{eqpro},
for any $k \geq 0$ we have
\[
\|p_k(x)   \o_d(x)  - \o_d(x)p^n_k(x) \| 
= \|(P^k \otimes P^n)  -  (I_d \otimes P^n)p^n_k\| = (k/(k+n))^{1/2}  
\]
for each $x$. 
\end{prop}
Crucially, this goes to 0 as $n \to \infty$, for fixed $k$.


\section{The core calculation for the case of $k \leq -1$}
\label{core2}

We now treat the case in which $k \leq -1$, for which we
must assume that $k+n\geq 1$. Again we set
\[
T^n_k = (P^k \otimes P^n)  -  (I_d \otimes P^n)p^n_k.
\]
We use the same basis vectors $e_i$ and $f_j$ 
for $\Hk$  and $\Hn$ as
in the previous section.
But now $k + n < n$, and $p^n_k$ is the projection on the
subspace of $\Hk \otimes \Hn$ generated by the 
highest weight vector $v_{k+n}$ of weight $k+n$. This
vector has a more complicated expression in terms
of the basis vectors $e_i \otimes f_j$  than for the case
of $k \geq 0$. Specifically, $v_{k+n}$ will be a linear
combination of those basis vectors  $e_i \otimes f_j$
that are of weight $k+n$. 

\begin{lem}
\label{highwt}
For $k \leq -1$ a highest weight vector of weight $k+n$ 
in $\Hk \otimes \Hn$
is given by 
\[
v_{k+n} 
= \sum _{b  = 0}^{-k} \a_b \  e_{-k-2b} \otimes f_{n + 2k + 2b} \ ,
\]
where
\[
\a_b = (-1)^b \frac{(n+k+b)!}{(n+k)!b!}
\]
for $0 \leq b \leq -k$. 
\end{lem}

\begin{proof}
We must determine the coefficients $\a_b$ for $v_{k+n}$
of the general form given in the statement of the lemma. 
In order for $v_{k+n}$ to be a
highest weight vector it must satisfy $Ev_{k+n} = 0$, that is, 
\[
0 = Ev_{k+n} 
= \sum _{b  = 0}^{-k} \a_b \  (Ee_{-k-2b} \otimes f_{n + 2k + 2b}
+e_{-k-2b} \otimes Ef_{n + 2k + 2b}).
\]
For each $b$ with $0 \leq b < -k$ the term in this sum that is
a multiple of 
\[
 e_{-k-2b} \otimes f_{n + 2k + 2(b+1)}
\]
is
\[
 \a_{b+1} \  Ee_{-k-2(b+1)} \otimes f_{n + 2k + 2(b+1)}
\ + \ a_b \ e_{-k-2b} \otimes Ef_{n + 2k + 2b} \ .
\]
By equation \ref{EFa} in the Appendix, 
\[
Ee_{-k-2(b+1)} =  (b+1)(-k-(b+1) +1)e_{-k-2(b+1)+2}   ,
\]
while
\[
Ef_{n + 2k + 2b} = Ef_{n - 2(-k  - b)}
= (-k-b)(n-(-k-b) + 1)f_{n + 2k + 2(b+1)}  \ .
\]
It follows that 
\[
0 = \a_{b+1}(b+1)(-k-b) \ + \ \a_b(-k-b)(n+k+b+1).
\]
Thus for $0 \leq b \leq -k - 1$ we have
\[
\a_{b+1} = -(n+k+b+1)(b+1)^{-1}\a_b \ ,
\]
that is, if $1 \leq b \leq -k$ then
\[
\a_b = -(n+k+b)b^{-1}\a_{b-1}   .
\]
We are free to set $\a_0 = 1$. On doing that, we find by
induction that
\[
\a_b = (-1)^b \frac{(n+k+b)!}{(n+k)!b!}
\]
for $0 \leq b \leq -k$. 
\end{proof}

Much as in the case in which $k \geq 0$, we set
\[
v_{k+n-2a} = F^a v_{k+n}
\]
for $a = 0, 1, \cdots, n+k$.  These vectors form an
orthogonal basis for a sub-representation of
$(\cH^{|k|} \otimes \cH^n, U^{|k|} \otimes U^n)$ that is
unitarily equivalent to the irreducible representation
$(\cH^{k+n}, U^{k+n})$, and we will identify the
latter with this sub-representation. 
The span of these vectors is, by definition, the range of
the projection $p^n_k$. 

We seek to determine $\|T^n_k\|$, and to show that, for
fixed $k$, it goes to 0 as $n$ goes to $\infty$. Recall
that $P^k$ is the rank-one projection on $e_k$, which
for $k \leq -1$ is the \emph{lowest} weight vector
in $\Hk$ (and is not of norm 1). Since the range of
$P^k \otimes P^n$ is spanned by $e_k \otimes f_n$
whereas the only vectors in the range of $p^n_k$ 
that are of weight $k + n$ are multiples of $v_{k+n}$, 
it is clear that the range of $P^k \otimes P^n$ is not
 included in the range of $p^n_k$, in contrast to what
 happens for $k \geq 0$. Let $W$ be the subspace of
 $\cH^{|k|} \otimes \cH^n$ spanned by the vectors
 $v_{k+n}, \cdots, v_{-k-n}$ together with $e_k \otimes f_n$.
 If $u$ is any vector in $\cH^{|k|} \otimes \cH^n$ that is
 orthogonal to $W$, then both $P^k \otimes P^n$ and
 $p^n_k$ take $u$ to 0, and thus so does $T^n_k$. 
 Consequently in order to determine $\|T^n_k\|$ it
 suffices to view $T^n_k$ as an operator from $W$
 into $\cH^{|k|} \otimes \cH^n$.
 
 We consider now the action of $T^n_k$ on the vectors
 $v_{k+n-2a}$. Let us assume first that $a \geq 1$. 
 Since each term in the
 formula for $v_{k+n-2a}$ must involve an elementary tensor of
 weight $k+n-2a$, it is clear that 
 $(P^k \otimes P^n)(v_{k+n-2a}) = 0$ for $a \geq 1$,
 and so
 \[
 T^n_k(v_{k+n-2a}) = 
 (I_d \otimes P^n)(F^a(v_{k+n})).
 \]
Since $F$ lowers weights, the only term in the formula for  $v_{k+n}$
given in Lemma \ref{highwt} on which $(I_d \otimes P^n)F^a$ has
 a possibility of being non-zero is the term for $b = -k$, that is 
 $\a_{-k} \ e_k \otimes f_n$. But because $e_k$ is the lowest weight
 vector in $\Hk$, we see that 
 $F(e_k\otimes f_n) = e_k \otimes f_{n-2}$, which is in
 the kernel of $I_d \otimes P^n$. We conclude that for all
 $a \geq 1$ we have $T^n_k(v_{k+n-2a}) = 0$.

Thus it suffices to determine the norm of the restriction of $T^n_k$
to the subspace spanned by $v_{k+n}$ and $e_k \otimes f_n$. Now
\begin{align*}
T^n_k(v_{k+n})  &= (P^k \otimes P^n)(v_{k+n}) 
 \ - \ (I_d \otimes P^n)p^n_k(v_{k+n})    \\
&= ((P^k - I_d)\otimes I_{n+1})(I_d \otimes P^n)(v_{k+n})   \\
&= ((P^k - I_d)\otimes I_{n+1})(\a_{-k} e_k \otimes f_n) \ = \ 0  .
\end{align*}

Thus, finally, it comes down to determining $T^n_k(e_k \otimes f_n)$.
Now clearly
\[
T^n_k(e_k \otimes f_n) = e_k \otimes f_n \ - \ 
(I_d \otimes P^n)p^n_k(e_k \otimes f_n).
\]
The weight vectors $v_{k+n-2a}$ form an orthogonal basis for the range
of $p^n_k$, and all of these vectors except the one for $a = 0$ are
of different weight than the weight of $e_k \otimes f_n$ and so are
orthogonal to $e_k \otimes f_n$. It follows that
\[
p^n_k(e_k \otimes f_n) \ = \ 
\frac{\langle e_k \otimes f_n, \ v_{k+n}\rangle}
{\|v_{k+n}\|^2} v_{k+n}   .
\]
But from the formula for  $v_{k+n}$
given in Lemma \ref{highwt} we see that
\[
(I_d \otimes P^n)(v_{k+n}) \ = \ \a_{-k} \  e_k \otimes f_n \ = \ 
\frac{\langle e_k \otimes f_n, \ v_{k+n} \rangle} {\|e_k \otimes f_n\|^2}
e_k \otimes f_n   .
\]
Thus 
\[
T^n_k(e_k \otimes f_n) = 
\big(1 \ - \  \frac{\langle e_k \otimes f_n, \ v_{k+n} \rangle^2}
{\|e_k \otimes f_n\|^2\|v_{k+n}\|^2}\big)
e_k \otimes f_n   , 
\]
so that 
\[
    \|T^n_k\| = \frac{\|v_{k+n}\|^2 \ - \ \langle \frac{e_k \otimes f_n}
    {\|e_k \otimes f_n\|},   v_{k+n}\rangle^2}
    {\|v_{k+n}\|^2}
 = \|v'_{k+n}\|^2/\|v_{k+n}\|^2   ,
\]
where $v'_{k+n}$ denotes $v_{k+n}$ with its last term
(involving $e_k \otimes f_n$) removed.
We want to show that the above expression
goes to 0 as $n \to \infty$  .

Now 
\[
\|v_{k+n}\|^2 = \sum _{b  = 0}^{-k} \a_b^2 \ 
 \|e_{-k-2b}\|^2  \|f_{n - 2(-k -b)}\|^2,
\]
while $\|v'_{k+n}\|^2$ is the same sum but with
 the upper limit of summation
being $-k-1$. Since each $\|e_{-k-2b}\|$ is
independent of $n$, to show that 
$\|v'_{k+n}\|^2/\|v_{k+n}\|^2$ converges
to 0 as $n \to \infty$, it suffices to show that for each $b$
with $0 \leq b \leq -k-1$ the term
\[
 \a_b^2 \  \|f_{n - 2(-k -b)}\|^2/\|v_{k+n}\|^2
\]
goes to 0 as $n \to \infty$. (Note that $\a_b$ does
depend on $n$.) To show this it suffices to show that
this holds when $v_{k+n}$ is replaced by the $b=-k$ term
in its expansion, which is the term missing in $v'_{k+n}$.
On noting that $e_k$ is independent of $n$ and that
$\|f_n\| = 1$ for all $n$, we see that we must
show that for each $b$ with $0 \leq b \leq -k-1$  
the term
\[
\a_b^2 \  \|f_{n - 2(-k -b)}\|^2/\a_{-k}^2
\]
goes to 0 as $n \to \infty$.

From the formula in Lemma \ref{highwt} we find that
for each $b$ with $0 \leq b \leq -k-1$ we have
\[
|\a_b|/|\a_{-k}| =  \frac{(n+k+b)!}{(n+k)!b!}/ 
\frac{n!}{(n+k)!(-k)!}
= \frac{(-k)!}{b!} \frac {(n+k+b)!}{n!}  \  ,
\]
while from formula \ref{keyeq} of the Appendix we have
\[
\|f_{n - 2(-k -b)}\|^2 = \frac{(-k-b)!n!}{(n+k+b)!}  \ \   .
\]
Thus
\begin{align*}
\a_b^2 \  \|f_{n - 2(-k -b)}\|^2/\a_{-k}^2 & =
\big(\frac{(-k)!}{b!}\big)^2 (-k-b)!\big(\frac{(n+k+b)!}{n!}\big)^2
\frac{n!}{(n+k+b)!}   \\
& \leq (-k)^3\frac{(n+k+b)!}{n!} \leq \frac{(-k)^3}{n}
\end{align*}
since $(k+b) \leq -1$. The power of $n$ can not be
improved, as seen by considering the case $b = -k-1$.

We conclude from these estimates that $\|T^n_k\| \to 0$
as $n \to \infty$, and consequently that:
\begin{prop}
\label{propcore2}
In terms of our earlier notation in Section \ref{secbp}, 
and by equation \ref{eqpro},
for any fixed $k \leq -1$ 
\[
\|p_k(x)   \o_d(x)  - \o_d(x)p^n_k(x) \| 
= \|(P^k \otimes P^n)  -  (I_d \otimes P^n)p^n_k\|
\]
converges to 0 as $n \to \infty$.
\end{prop}

I have not managed to extend the results of this 
section to the case of general coadjoint orbits
of compact semisimple Lie groups.


\section{The main theorem and its proof}
\label{secproof}

We are now in position to state and prove
the main theorem of this paper.
 We will recall some of the
notation at the beginning of the proof.

\begin{thm}
\label{mainthm}
Let notation be as above, for $G = SU(2)$ and a chosen 
continuous length
function $\ell$, etc. Fix the integer $k$ (and set $d=|k|+1$). 
Let $L^n_k$ be 
defined as in equation \eqref{eqlr} for $r = r^n = l_{\Pi^n_d}$. Then
$L^n_k(p_k, \ p^n_k)$ goes to 0 as $n \to \infty$. Furthermore,
we can find a natural number $N_k$ large enough that
for every $n \geq N_k$ we have
\[
(h_{\Pi^n_d} +l_{\Pi^n_d})L^n_k(p_k, \ p^n_k) < 1/2 \ ,
\]
so that if $q$ is any other projection in $M_d(\cB^n)$
that satisfies this same inequality when $p^n_k$ is 
replaced by $q$, then there is a continuous path
of projections in $M_d(\cB^n)$ going from $p^n_k$
to $q$, which implies that the projective $\cB^n$ modules 
determined by $p^n_k$
and $q$ are isomorphic. In this sense the projective
$\cB^n$-module $\O^n_k$ is associated by the
bridge $\Pi^n_d$ to the projective $\cA$-module
$\Xi_k$.
\end{thm}

\begin{proof}
We recall some of our earlier notation and results. We
have $\cA = C(G/H)$ and $\cB^n = \cL(\Hn)$. Then we
let $\cD^n = \cA\otimes \cB^n$, and we let
$\Pi^n = (\cD^n, \o^n)$, a bridge from $\cA$ to $\cB^n$.
Let $r_{\Pi^n}$ denote the reach of the bridge $\Pi^n$ as
measured by $L^\cA$ and $L^{B^n}$, as
defined in Definition \ref{reach}. 
Define a seminorm, $N_{\Pi^n}$, and then a $*$-seminorm
$\hat N_{\Pi^n}$
on $\cA \oplus \cB^n$ much as done
just before equation \ref{eqlr},
and then define, for some $r^n \geq r_{\Pi^n}$, 
a seminorm $L^n_{r^n}$ on $\cA \oplus \cB^n$ by 
\[
L^n_{r^n}(a, b) = 
L^\cA(a) \vee L^\cB(b) \vee (r^n)^{-1}\hat N_{\Pi^n}(a, b)  .
\]
Then $L^n_{r^n}$ is an admissible seminorm on $\cA \oplus \cB^n$, 
according to Proposition \ref{admis}.

We need the matricial version of this seminorm.
We set
$\Pi^n_d = (M_d(\cD^n), \o^n_d)$, where $\o^n_d$
is defined, much as in Definition \ref{defmat}, by
\[
\o^n_d(x) = I_d \otimes \a_x(P^n)
\]
for all $x \in G$.
Then
$\Pi^n_d$ is a bridge from $M_d(\cA)$ to $M_d(\cB^n)$. We
measure it with the seminorms $L^\cA_d$ and $L^{\cB^n}_d$,
defined much as at the end of 
Section \ref{secalg}. We now denote the resulting
reach and height of $\Pi^n_d$ by $r_{\Pi^n_d}$ and $h_{\Pi^n_d}$. 

Define,  on $M_d(\cA)\oplus M_d(\cB^n)$, a seminorm, $N^n_d$,
by $N^n_d(a,b) = \|a\o_d^n - \o_d^n b\|$, 
and then a $*$-seminorm
$\hat N^n_d$, much as done
just before equation \ref{eqlr}.
Then, for any $r^n \geq r_{\Pi^n_d}$
define a seminorm, $L^n_{d,r^n}$, by
\begin{equation}
\label{eqLn}
L^n_{d,r^n}(a,b) = L^\cA_d(a) \vee L^{\cB^n}_d(b)
\vee (r^n)^{-1} \hat N^n_d(a,b).
\end{equation}
Then $L^n_{d,r^n}$ is an admissible seminorm for 
$L^\cA_d$ and $L^{\cB^n}_d$, by Proposition \ref{admis}, because
$r^n \geq r^n_d$.

Let $p_k$ and $p^n_k$ be the projections defined 
in Notation \ref{notproj} for
the projective modules $\Xi_k$ and $\O^n_k$. Then
\[
L^n_{d,r^n}(p_k, p^n_k) = L^\cA_d(p_k) \vee L^{\cB^n}_d(p^n_k)
\vee (r^n)^{-1} \hat N^n_d(p_k, p^n_k).
\]
According to part a) of Theorem \ref{th4.5}, in order for
$p^n_k$ to be a projection associated to $p_k$ up
to path connectedness, we need that
\[
(h_{\Pi^n_d} + r^n) L^n_{d,r^n}(p_k,p^n_k) \ < \ 1/2.
\]
Thus each of the three main terms in the formula for
$L^n_{k,r^n}(p_k,p^n_k)$ must satisfy the corresponding inequality.

We examine the third term first. Because $p_k$ and $p^n_k$
are self-adjoint, this term is equal to 
\begin{align*}
(h_{\Pi^n_d} + r^n)(r^n)^{-1}\|p_k(x)   \o_d(x)  - \o_d(x)p^n_k(x) \| \ .
\end{align*}
We now use theorem 6.10 of \cite{R29} (where $q$ there 
is our $k$, and $m$ is our $n$), which is one of the 
two main theorems of \cite{R29}. It tells us 
the quite un-obvious fact that  $l_{\Pi^n_d}$,
and so both $r_{\Pi^n_d}$ and $h_{\Pi^n_d}$, 
go to $0$ as
$n \to \infty$, for fixed $k$. 
This theorem furthermore gives quantitative upper bounds for
$l_{\Pi^n_d}$ in terms of the length function $\ell$ chosen for $G$.

We also now see the reason for allowing
$r^n$ to possibly be different from $r_{\Pi^n_d}$ in 
defining $L^n_{d,r^n}$, 
namely that if $r_{\Pi^n_d}$ goes to
$0$ more rapidly than does $h_{\Pi^n_d}$, then their ratio
goes to $+\infty$, so that the term $(h_{\Pi^n_d} + r^n)(r^n)^{-1}$
in the displayed expression above goes to $+\infty$.
There are many ways to choose $r^n$
to avoid this problem, but the simplest is probably just to
set $r^n = \max\{r_{\Pi^n_d}, h_{\Pi^n_d}\}$, which is just the 
definition of $l_{\Pi^n_d}$.
We now make this choice, and for this choice we
write $L^n_d$ instead of $L^n_{d,{r^n}}$. Then we have
$1 + h_{\Pi^n_d}/l_{\Pi^n_d}  \ \leq \ 2$, and so
\[
(h_{\Pi^n_d} + l_{\Pi^n_d})
(l_{\Pi^n_d})^{-1}\|p_k   \o^n_d  - \o^n_d p^n_k\|     
\leq  2\|p_k   \o^n_d  - \o^n_d p^n_k\| .
\]
From Propositions \ref{propcore1} and \ref{propcore2}
it follows that for fixed $k$ this term goes to $0$ as $n$
goes to $\infty$.

We examine next the first term, 
$
(h_{\Pi^n_d} + l_{\Pi^n_d})L^\cA_d(p_k) .
$
Since $L^\cA_d(p_k)$ is independent of $n$, and we have seen
that $(h_{\Pi^n_d} + l_{\Pi^n_d})$ goes to $0$ as $n \to \infty$,
it follows that this first term too goes to $0$ as $n \to \infty$, 
for fixed $k$.

Finally, we examine the second term,
\[
(h_{\Pi^n_d} + l_{\Pi^n_d})L^{\cB^n}_d(p^n_k) .
\]
In examining above the first term we 
have seen that $h_{\Pi^n_d}$ and $l_{\Pi^n_d}$ go
to 0 as $n \to \infty$, so the
only issue is the growth of $L^{\cB^n}_d(p^n_k)$ as $n \to \infty$.
But Proposition \ref{pknbound} tells us exactly that, for fixed $k$, 
there is a common bound for the $L^{\cB^n}_d(p^n_k)$'s.

We now apply Theorem \ref{th4.5}  to the present 
situation, and this concludes the proof.
\end{proof}

As mentioned at the end of Section \ref{secalg}, 
upper bounds for $h_{\Pi^n_d}$ 
and $l_{\Pi^n_d}$ in terms of just the data for
$\Pi$, $L^\cA$, and $L^\cB$ are given in
theorem 5.5 of \cite{R29}.


\section{Bridges and direct sums of projective modules}
\label{secsums}

In this section we discuss how to deal with direct sums of 
projective modules, and we indicate in what sense it is sufficient 
for us to deal in detail here only with the line-bundles on the
2-sphere.

It is well known that every complex vector bundle over
the 2-sphere is isomorphic to the direct sum of a
line bundle with a trivial bundle. (Use e.g. Proposition 1.1 of
chapter 8 of \cite{Hsm}.) We can see this in part as follows.

\begin{prop}
\label{prpsum}
With notation as in Notation \ref{amod}, for any $j,k \in \bZ$ we
have a natural module isomorphism
\[
\Xi_j \oplus \Xi_k  \cong  \ \Xi_{j+k} \oplus \Xi_0   .
\]
\end{prop}

\begin{proof}
To simplify notation, we identify $SU(2)$ with the 3-sphere in
the usual way, so that $(z,w) \in S^3 \subseteq \bC^2$
corresponds to the matrix 
$(\begin{smallmatrix}
z & -\bar w     \\
w & \bar z
\end{smallmatrix})$.
If we set $e_t = e(t)$ for each $t \in \bR$, then right 
multiplication of elements of $SU(2)$ by the matrix
$(\begin{smallmatrix}
e_t & 0     \\
0 & \bar e_t
\end{smallmatrix})$ 
corresponds to sending $(z,w)$
to $(ze_t,we_t)$. Thus elements of $\Xi_k$ can be viewed
as continuous functions $\xi$ on $S^3$ that satisfy
\[
\xi(ze_t,we_t) = \bar e_t^k \xi(z,w)
\]
for all $t \in \bR$. 

Let $j, k \in \bZ$ be given. Define a $GL(2, \bC)$-valued 
function $M$ on $S^3$ by 
\[
M(z,w) = 
\begin{pmatrix}
\bar z^k & -\bar w^j   \\
 w^j &  z^k
\end{pmatrix}.
\]
For any
$(\begin{smallmatrix}
f     \\
g
\end{smallmatrix})
\in \Xi_j \oplus \Xi_k$ 
set $\Phi 
(\begin{smallmatrix}
f     \\
g
\end{smallmatrix})
= M 
(\begin{smallmatrix}
f     \\
g
\end{smallmatrix})$,
so that 
\[
\Phi 
\begin{pmatrix}
f     \\
g
\end{pmatrix}(z,w) = 
\begin{pmatrix}
\bar z^k f(z)  - \bar w^j g(w)     \\
 w^j f(z) +  z^k g(w)    .
\end{pmatrix}
\]
It is easily checked that 
$\Phi 
(\begin{smallmatrix}
f     \\
g
\end{smallmatrix})$
is in $\Xi_{j+k} \oplus \Xi_0$, and that
$\Phi$ is an $\cA$-module homomorphism.
Furthermore, $\Phi$ has an inverse, obtained by using the
inverse of $M$. Thus $\Phi$ is an isomorphism, as needed.
\end{proof}

This proposition can be used inductively to show that the
direct sum of any finite number of the $\Xi_k$'s is isomorphic
to the direct sum of a single $\Xi_j$ with a trivial (i.e. free)
module.

The corresponding result for the $\Onk$'s is even
easier:

\begin{prop}
\label{prpsms}
With notation as in Notation \ref{bmod}, for any $j,k \in \bZ$
and for any $n \geq 1$ with $n+j \geq 0$, 
$n+k \geq 0$, and $n+j+k \geq 0$, we
have a natural module isomorphism
\[
\O^n_j \oplus \O^n_k  \cong  \ \O^n_{j+k} \oplus \O^n_0   .
\]
\end{prop}

\begin{proof}
We have
\begin{align*}
\O^n_j \oplus \O^n_k  &=  \cL(\cH^n, \cH^{j+n}) 
\oplus  \cL(\cH^n, \cH^{k+n})    \\
& \cong \cL(\cH^n, \cH^{j+k+2n}) 
\cong  \cL(\cH^n, \cH^{j+k+n}) 
\oplus  \cL(\cH^n, \cH^{n})     \\
& = \O^n_{j+k} \oplus \O^n_0 . 
\end{align*}
\end{proof}

But if one wants to show these correspondences by using
projections associated to the projective modules, there
are substantial complications. Let us consider the general case first.

Let $\cA$ and $\cB$ be unital C*-algebras, and let $\Pi = (\cD, \o)$ be 
a bridge from $\cA$ to $\cB$. Let $\{L^\cA_n\}$ and
$\{L^\cB_n\}$ be matrix Lip-norms on $\cA$ and $\cB$ 
(as defined in Definition \ref{defmlip}). For a given
$d$ we can use $L^\cA_d$ and
$L^\cB_d$ to measure the length of $\Pi_d$. 
One significant difficulty
is that it seems to be hard in general 
to obtain an upper bound for the
length of $\Pi_d$ in terms of the length of $\Pi$ (though
we will see that
for the case of the ``bridges with conditional expectation''
that are discussed in \cite{R29} we can get some useful
information). But the following little result will be
useful below.

\begin{prop}
\label{prpineq}
Let $\Pi = (\cD, \o)$ be 
a bridge from $\cA$ to $\cB$, and let
$\{L^\cA_n\}$ and
$\{L^\cB_n\}$ be matrix slip-norms on $\cA$ and $\cB$,
used to measure the length of $\Pi_d$ for any $d$. 
If $e$ is another natural number such that
$d < e$ then $r_{\Pi_d}  \leq  r_{\Pi_e}$. 
\end{prop}

\begin{proof}
Let $A \in M_d(\cA)$ with $A^* = A$ and $L^\cA_d(A) \leq 1$,
so that $A \in \cL^1_d(\cA)$. Then 
$(\begin{smallmatrix}
A & 0 \\
0 & 0
\end{smallmatrix})$, with the $0$'s of correct sizes, is 
in $\cL^1_e(\cA)$ by property (2) of Definition \ref{defmtx}.
Let $\d > 0$ be given. Then by the definition of $r_{\Pi_e}$
there is a $C \in \cL^1_e(\cB)$ such that 
\[
\|
\begin{bmatrix}
A & 0 \\
0 & 0
\end{bmatrix}
\begin{bmatrix}
\o_d & 0 \\
0 & \o_{e-d}
\end{bmatrix}
\ - \
\begin{bmatrix}
\o_d & 0 \\
0 & \o_{e-d}
\end{bmatrix}
C\| \leq r_{\Pi_e} + \d   .
\]
Compress the entire term inside the norm symbols by the
matrix $E = 
(\begin{smallmatrix}
I_d & 0 \\
0 & 0
\end{smallmatrix})$, and define $B$ by $ECE = 
(\begin{smallmatrix}
B & 0 \\
0 & 0
\end{smallmatrix})$ to obtain
\[
\|A\o_d - \o_d B\| \leq r_{\Pi_e} + \d   .
\]
Note that $B^* = B$ and that $L^\cB_d(B) \leq 1$
by property (1) of Definition \ref{defmtx}, so that
$B \in \cL^1_d(\cB)$. Since $\d$ is arbitrary, it follows
that the distance from $A\o_d$ to $\o_d \cL^1_d(\cB)$
is no greater than $r_{\Pi_e}$. In the same way we
show that for any $B \in \cL^1_d(\cB)$ there is an
$A \in \cL^1_d(\cA)$ such that the distance 
from $\o_d B$ to $ \cL^1_d(\cA)\o_d$
is no greater than $r_{\Pi_e}$.
It follows that $r_{\Pi_d}  \leq  r_{\Pi_e}$ as desired.
\end{proof}

For any natural number $d$ 
let $N_d$ be the seminorm on $M_d(\cA) \oplus M_d(\cB)$
defined much as done shortly before equation \ref{eqlr}, by
\[
N_d(A, B) = \|A\o_d - \o_d B\|
\]
for $A \in M_d(\cA)$ and $B \in M_d(\cB)$. Suppose now
that for natural numbers $d$ and $e$ we have
$a_1 \in M_d(\cA)$ and $a_2 \in M_e(\cA)$ as well as
$b_1 \in M_d(\cB)$ and $b_2 \in M_e(\cB)$, so that
\[
\begin{bmatrix}
a_1 & 0 \\
0 & a_2
\end{bmatrix}
\in M_{d+e}(\cA)
\quad \quad \mathrm{and} \quad \quad
\begin{bmatrix}
b_1 & 0 \\
0 & b_2
\end{bmatrix}
\in M_{d+e}(\cB)   .
\]
Then
\begin{align*}
N_{d+e}(
\begin{bmatrix}
a_1 & 0 \\
0 & a_2
\end{bmatrix}
 \ , \ 
\begin{bmatrix}
b_1 & 0 \\
0 & b_2
\end{bmatrix}
) &=
\| \begin{bmatrix}
a_1\o_d - \o_d b_1 & 0 \\
0 & a_2 \o_e - \o_e b_2
\end{bmatrix} \|      \\
&= N_d(a_1, b_1) \vee N_e(a_2, b_2)   .
\end{align*}
Next, much as done shortly before equation \ref{eqlr},
for each $d$ 
we define a $*$-seminorm, $\hat N_d$, by
\[
\hat N_d(A, B) = N_d(A,B) \vee N_d(A^*, B^*)   ,
\]
for $A \in M_d(\cA)$ and $B \in M_d(\cB)$. It then follows
from the calculation done just above that 
\begin{align*}
\hat N_{d+e}(
\begin{bmatrix}
a_1 & 0 \\
0 & a_2
\end{bmatrix}
 \ , \ 
\begin{bmatrix}
b_1 & 0 \\
0 & b_2
\end{bmatrix}
) 
= \hat N_d(a_1, b_1) \vee \hat N_e(a_2, b_2)   ,
\end{align*}
for $a_1 \in M_d(\cA)$ and $a_2 \in M_e(\cA)$, and
$b_1 \in M_d(\cB)$ and $b_2 \in M_e(\cB)$.

Next, assume that $r_{\Pi_d} < \infty$ for each $d$, as
is the case for matrix Lip-norms as seen in Proposition
\ref{propfr}.
Let some 
choice of finite $r_d \geq r_{\Pi_d}$ be given for each $d$.
Set, much as in equation \ref{eqlr},
\begin{equation*}
\label{eqlr2}
L^{r_d}_d(A,B) = 
L_d^\cA(A) \vee L_d^\cB(B) \vee r_d^{-1} \hat N_{\Pi_d}(A, B)   
\end{equation*}
for $A \in M_d(\cA)$ and $B \in M_d(\cB)$. It then follows
from the calculations done above that 
\[
L^{r_{d+e}}_{d+e}(
\begin{bmatrix}
a_1 & 0 \\
0 & a_2
\end{bmatrix}
 \ , \ 
\begin{bmatrix}
b_1 & 0 \\
0 & b_2
\end{bmatrix}
) 
= L^{r_{d+e}}_d(a_1, b_1) \vee L^{r_{d+e}}_e(a_2, b_2)   ,
\]
for $a_1 \in M_d(\cA)$ and $a_2 \in M_e(\cA)$, and
$b_1 \in M_d(\cB)$ and $b_2 \in M_e(\cB)$. Because
$r_{\Pi_{d+e}} \geq r_{\Pi_d} \vee r_{\Pi_e}$ according
to Proposition \ref{prpineq}, one can check quickly 
that  $L^{r_{d+e}}_d$ is admissible (Definition \ref{admisss}) 
for $L^\cA_d$ 
and $L^\cB_d$, and similarly for $L^{r_{d+e}}_e$.

Suppose now that $p_1 \in M_d(\cA)$ and $p_2 \in M_e(\cA)$
are projections representing projective $\cA$-modules, and
that $q_1 \in M_d(\cB)$ and $q_2 \in M_e(\cB)$ are projections
representing projective $\cB$-modules. Then
\begin{equation*}
L^{r_{d+e}}_{d+e}(
\begin{bmatrix}
p_1 & 0 \\
0 & p_2
\end{bmatrix}
 \ , \ 
\begin{bmatrix}
q_1 & 0 \\
0 & q_2
\end{bmatrix}
) 
= L^{r_{d+e}}_d(p_1, q_1) \vee L^{r_{d+e}}_e(p_2, q_2)   .
\end{equation*}
From Theorem \ref{th4.5}a we then obtain:

\begin{prop}
\label{prpvee}
Let notation be as just above, and assume that $l_{\Pi_{d+e}} < \infty$. 
If
\[
(h_{\Pi_{d+e}} + r_{d+e})
\max\{(L^{r_{d+e}}_d(p_1, q_1), L^{r_{d+e}}_e(p_2, q_2))\} < 1/2    ,
\]
and if there is a projection $Q \in M_{d+e}(\cB)$ such
that 
\begin{equation*}
\label{eqsum}
(h_{\Pi_{d+e}} + r_{d+e})
L^{r_{d+e}}_{d+e}(
\begin{bmatrix}
p_1 & 0 \\
0 & p_2
\end{bmatrix}
 \ , \ 
Q) < 1/2    ,
\end{equation*}
then there is a path through projections in 
$M_{m+n}(\cB)$ going from $Q$ to 
$(\begin{smallmatrix}
q_1 & 0 \\
0 & q_2
\end{smallmatrix})
   .$
\end{prop}

This uniqueness result means that $(\begin{smallmatrix}
q_1 & 0 \\
0 & q_2
\end{smallmatrix})$
is a projection in $M_{d+e}(\cB)$ corresponding to the
projection
$(\begin{smallmatrix}
p_1 & 0 \\
0 & p_2
\end{smallmatrix})$   .
Thus the consequence of this proposition is
that for suitable bounds, if the projective modules
for $p_1$ and $q_1$ correspond, and if those for
$p_2$ and $q_2$ correspond, then the direct sum
modules correspond. This suggests a further reason
for saying that when
considering complex vector bundles over the 2-sphere
it suffices for our purposes to consider only the 
line bundles. 

The difficulty with using this proposition is that I have not
found a good way 
of bounding in general the reach and height of $\Pi_d$ in terms of
of those of $\Pi$. However, the specific bridges that we
have been using for matrix algebras converging to the
sphere are examples of ``bridges with conditional
expectations'', as defined in \cite{R29}. For
such a bridge, bounds are obtained in terms
of the conditional expectations. More specifically,
there is a constant, $\mathring{\g}_\Pi$, (equal to 
$2 \max \{\g^\cA, \g^\cB\}$ in the notation of 
theorem 5.4 of \cite{R29}) such 
that 
\[
r_{\Pi_d} \leq d \mathring{\g}_\Pi
\] 
for all $d$,
and there is a constant, $\mathring{\d}_\Pi$,
(which in the notation of 
theorem 5.4 of \cite{R29} is
equal to $2 \max \{ \min\{\d^\cA, \hat \d^\cA\}, 
\min \{\d^\cB, \hat \d^\cB\}$  ) such that
\[
h_{\Pi_d} \leq d \mathring{\d}_\Pi
\]
for all $d$.
Then if we choose $s \geq  \mathring{\g}_\Pi$ we have
\[
(d+e)s \geq (d+e) \mathring{\g}_\Pi 
\geq r_{\Pi_{d+e}}   ,
\]
so we can set $r_{d+e} = (d+e)s$ and apply
Proposition \ref{prpvee} to obtain:

\begin{prop}
\label{prpexp}
Assume that $\Pi$ is a bridge with conditional expections, 
and let notation 
be as just above. Choose $s \geq \mathring{\g}_\Pi$. 
If
\[
(d+e)(\mathring{\d}_\Pi + s)
\max\{(L^{(d+e)s}_d(p_1, q_1), L^{(d+e)s}_e(p_2, q_2))\} < 1/2    ,
\]
and if there is a projection $Q \in M_{d+e}(\cB)$ such
that 
\begin{equation}
\label{eqexp}
(d+e)(\mathring{\d}_\Pi + s)
L^{(d+e)s}_{d+e}(
\begin{bmatrix}
p_1 & 0 \\
0 & p_2
\end{bmatrix}
 \ , \ 
Q) < 1/2    ,
\end{equation}
then there is a path through projections in 
$M_{d+e}(\cB)$ going from $Q$ to 
$(\begin{smallmatrix}
q_1 & 0 \\
0 & q_2
\end{smallmatrix})
   .$
\end{prop}

Thus again, if the inequality \ref{eqexp} is satisfied,
then  if the projective modules
for $p_1$ and $q_1$ correspond, and if those for
$p_2$ and $q_2$ correspond, then the direct sum
modules correspond. But notice that the factor
$d+e$ at the beginning means that as $d$ 
or $e$ get bigger, the remaining term must
be smaller in order for the product to be
$< 1/2$. Thus, for example, suppose that
$p$ and $q$ correspond. This will not in
general imply that 
$(\begin{smallmatrix}
p & 0 \\
0 &  0_e
\end{smallmatrix})$  
and
$(\begin{smallmatrix}
q & 0 \\
0 &  0_e
\end{smallmatrix})$ 
correspond, 
where $0_e$ denotes the $0$ matrix of size $e$.

We can now apply the above results to our basic
example in which $\cA = C(G/H)$ and $\cB^n = \cL(\cH^n)$,
and we have the bridge $\Pi^n$ between them, 
as in the previous section
and earlier.
From the discussion leading to propositions 6.3
and 6.7 of \cite{R29}, which is strongly based on
the results of \cite{R21}, it can be seen that
the constants $\mathring{\g}_{\Pi^n}$ 
and $\mathring{\d}_{\Pi^n}$ converge to 0
as $n \to \infty$. If in Proposition \ref{prpexp}
one sets $s_n=\max \{ \mathring{\g}_{\Pi^n},
\mathring{\d}_{\Pi^n} \}$, then for fixed
$d$ and $e$ the term
$(d+e)(\mathring{\d}_\Pi^n + s_n)$ 
in inequality \ref{eqexp} will converge
to 0 as $n \to \infty$. Consequently, for projections
$p_1$ and $p_2$ as above, one can find a 
sufficiently large $N$ that if $n \geq N$
and if $q_1$ and $q_2$ are corresponding
projections in $M_d(\cB^n)$ and $M_e(\cB^n)$ 
respectively, then inequality \ref{eqexp} is
satisfied, so that 
$(\begin{smallmatrix}
p_1 & 0 \\
0 & p_2
\end{smallmatrix})$  
and 
$(\begin{smallmatrix}
q_1 & 0 \\
0 & q_2
\end{smallmatrix})$  
correspond.


\section{Appendix. Weights}
\label{secapp1}
We give here a precise statement of the conventions we
use concerning weights, weight vectors, etc., and of their
properties, including some proofs (all basically well-known, e.g.
in section VIII.4 of \cite{Smn}).

At first we assume only that $\cH$ is a finite-dimensional
vector space over $\bC$, and that $H$ is a 
non-zero operator
on $\cH$ while $E$ and $F$ are operators on $\cH$
satisfying the relations 
\[
[E,F] = H, \quad \quad [H, E] = 2E, \quad \quad 
\mathrm{and} \quad [H, F] = -2F  .
\]
Let $\xi$ be an eigenvector of $H$ with eigenvalue $r$. Then
\[
H(E\xi) = E(H\xi) + [H,E]\xi = rE\xi + 2E\xi = (r+2)E\xi  ,
\]
so that if $E\xi \neq 0$ then $E\xi$ is an eiginvector 
for $H$ of eigenvalue r+2. In the same way, $F\xi$, if $\neq 0$,
is an eigenvector of $H$ of eigenvalue $r-2$. (So $E$
and $F$ are often called ``ladder operators''.) Since $\cH$
is finite-dimensional, 
it follows that there must be an eigenvector, $\xi_*$, for $H$
such that $E\xi_* = 0$ (a highest weight). 
Fix such a $\xi_*$, and let $r$ be its
eigenvalue. Then for each natural number $a$
we see that $F^a\xi_*$ will be an eigenvector of 
eigenvalue $r-2a$ unless $F^a\xi_* = 0$. Since $\cH$ is
finite-dimensional, there will be a natural number $m$ such
that $F^a\xi_* \neq 0$ for $a \leq m$ but $F^{m+1}\xi_* = 0$.
Now $E(F\xi_*) = [E, F]\xi_* = r\xi_*$. In a similar way we
find by induction that
\begin{align*}
E(F^a\xi_*) &= EF(F^{a-1}\xi_*) = (H+FE)(F^{a-1}\xi_*)  
 \\  \nonumber
&= a(r-a+1)F^{a-1}\xi_*.
\end{align*}
Consequently, since $F^{m+1}\xi_* = 0$, we have
\[
H(F^m\xi_*) = -F(EF^m\xi_*) = -m(r-m+1)F^m\xi_*.
\]
But also $H(F^m\xi_*) = (r-2m)F^m\xi_*$. 
Since $F^m\xi_* \neq 0$, it follows that $r = m$, and
so it is appropriate to denote $\xi_*$ by $e_m$. With this
notation, set
\[
e_{m-2a} = F^a e_m
\]
for $a = 0, \dots, m$. Each of these vectors is an
eigenvector of $H$ with 
corresponding eigenvalue $m-2a$. These
eigenvectors span a subspace of $\cH$ of dimension $m+1$. 
From the calculations done
above, it is clear that this subspace is carried into itself
by the operators $H, E$ and $F$, and that furthermore, 
we have
\begin{equation}
    E(e_{m-2a}) = a(m-a+1)e_{m -2a+2}  .  \label{EFa}
\end{equation}
From this we immediately obtain
\[
FE(e_{m-2a}) = a(m-a+1)e_{m-2a}  .
\]
Since $EF = H+FE$, it then follows that
\begin{equation}
EF(e_{m-2a}) = (a-1)(m-a)e_{m-2a}   .
\end{equation}

Now assume that $\cH$ is a finite-dimensional Hilbert space, and 
let $e_m$ be a highest weight vector of weight $m$
as above, and  
for each $a = 1, \dots, m$ define 
$e_{m-2a}$ as above. Assume now that
$F = E^*$ (the adjoint of $E$).
This implies that $H$ is
self-adjoint, so that its eigenvectors of different eigenvalue are
orthogonal. Consequently the $e_{m-2a}$'s 
are orthogonal vectors, that span a subspace of $\cH$ that
is carried into itself by the operators $H$, $E$, and $F$ (giving
an irreducible unitary representation of $SU(2)$). 
Assume further that $\|e_m\| = 1$. 
We need to know the
norms of the vectors $e_{m-2a}$. 
Notice that from equation \eqref{EFa} 
we see that for $a \geq 1$ we have
\begin{align*}
\|e_{m-2a}\|^2 &= \<Fe_{m-2a+2}, e_{m-2a}\> 
= \<e_{m-2a+2}, Ee_{m-2a}\>    \\
&= a(m-a+1)\|e_{m-2a+2}\|^2   .
\end{align*}

A simple induction argument then shows that
\begin{equation}
\|F^ae_m\|^2 = \|e_{m-2a}\|^2 = a!\Pi_{b = 0}^{a-1} (m-b)
= \frac{a!m!}{(m-a)!}   
\label{keyeq}
\end{equation}
for $a = 1, \dots, m$. 
(We find it convenient not to normalize the
$e_m$'s to length 1.)




\def\dbar{\leavevmode\hbox to 0pt{\hskip.2ex \accent"16\hss}d}
\providecommand{\bysame}{\leavevmode\hbox to3em{\hrulefill}\thinspace}
\providecommand{\MR}{\relax\ifhmode\unskip\space\fi MR }
\providecommand{\MRhref}[2]{%
  \href{http://www.ams.org/mathscinet-getitem?mr=#1}{#2}
}
\providecommand{\href}[2]{#2}

\end{document}